\documentclass[a4paper,reqno, 10pt]{amsart}

\usepackage{amsmath,amssymb,amsfonts,amsthm, mathrsfs}
\usepackage[latin1]{inputenc}
\usepackage{lmodern}

\usepackage{verbatim,wasysym,cite}

\usepackage[latin1]{inputenc}
\usepackage{microtype}
\usepackage{color,enumitem,graphicx}
\usepackage[colorlinks=true,urlcolor=blue, citecolor=red,linkcolor=blue,
linktocpage,pdfpagelabels, bookmarksnumbered,bookmarksopen]{hyperref}
\usepackage[english]{babel}
\usepackage[symbol]{footmisc}
\renewcommand{\epsilon}{{\varepsilon}}

\numberwithin{equation}{section}
\newtheorem{theorem}{Theorem}[section]
\newtheorem{lemma}[theorem]{Lemma}
\newtheorem{remark}[theorem]{Remark}

\newtheorem{proposition}[theorem]{Proposition}
\newtheorem{corollary}[theorem]{Corollary}

\newcommand{\C}{\mathbb C}

\newcommand{\R}{\mathbb R}
\newcommand{\N}{\mathbb N}

\def\({\left(}
\def\){\right)}
\def\<{\left\langle}
\def\>{\right\rangle}

\def\O{\mathcal O}

\def\F{\mathcal F}
\def\K{\mathcal K}

\def\EE{\mathcal E}

\def\eps{\varepsilon}

\DeclareMathOperator{\RE}{Re}
\DeclareMathOperator{\IM}{Im}

\newcommand{\qtq}[1]{\quad\text{#1}\quad}

\begin{document}

\title[3d NLS with combined terms]
{Threshold dynamics for the 3$d$ radial NLS with combined nonlinearity}

\author[Alex H. Ardila]{Alex H. Ardila}
\address{Department of Mathematics, Universidade Federal de Minas Gerais\\ ICEx-UFMG\\ CEP
30123-970\\ MG, Brazil} 
\email{ardila@impa.br}

\author[Jason Murphy]{Jason Murphy}
\address{Department of Mathematics \& Statistics, Missouri University of Science \& Technology, Rolla, MO, USA}
\email{jason.murphy@mst.edu}

\author[Jiqiang Zheng]{Jiqiang Zheng}
\address{Institute of Applied Physics and Computational Mathematics,  \ Beijing, \ 100088, \ China}
\email{zhengjiqiang@gmail.com}

\begin{abstract}
We consider the nonlinear Schr\"odinger equation with focusing quintic and defocusing cubic nonlinearity in three space dimensions:
\[
(i\partial_t+\Delta)u = |u|^2 u - |u|^4 u.
\]
In \cite{MiaoXuZhao2013, XuZhao2020}, the authors classified the dynamics of solutions under the energy constraint $E(u)< E^c(W)$, where $W$ is the quintic NLS ground state and $E^c$ is the quintic NLS energy.  In this work we classify the dynamics of $H^1$ solutions at the threshold $E(u)=E^c(W)$.
\end{abstract}


\maketitle

\medskip

\section{Introduction}
\label{sec:intro}
We consider the following cubic-quintic nonlinear Schr\"odinger equation (NLS):
\begin{equation}\label{NLS}
\begin{cases}
i\partial_{t}u+\Delta u=|u|^{2}u-|u|^{4}u,\\
u(0,x)=u_{0}(x)\in H_{x}^{1}(\R^{3}),
\end{cases}
\end{equation}
with $u:\R\times\R^3\to\C$. This is a Hamiltonian equation corresponding to the following conserved \emph{energy}, which is finite for $H^1$ data: 
\[
E(u)=\int_{\R^{3}}\big(\tfrac{1}{2}|\nabla u|^{2}-\tfrac{1}{6}|u|^{6}+\tfrac{1}{4}|u|^{4}\big)\,dx.
\]
Solutions additionally conserve the \emph{mass}, defined by
\[
M(u)=\tfrac{1}{2}\int_{\R^3} |u|^2\,dx.
\]

An important related model is the quintic (energy-critical) NLS
\begin{equation}\label{ECgNLS}
\begin{cases} 
(i\partial_{t}+\Delta) u+|u|^{4}u=0,\\
u(0)=u_{0}\in \dot{H}_{x}^{1}(\R^{3}),
\end{cases} 
\end{equation}
with associated energy
\[
E^{c}(u)=\int_{\R^{3}}\big(\tfrac{1}{2}|\nabla u|^{2}-\tfrac{1}{6}|u|^{6}\big)\,dx.
\]

The equation \eqref{ECgNLS} has static solutions of the form $u(t,x)=W(x)$, where $W$ solves the elliptic equation
\begin{align}\label{ellipC}
-\Delta W=|W|^{4}W.
\end{align}
An explicit solution, known as the \emph{ground state}, is given by 
\[
W(x):=\(1+ \tfrac{|x|^{2} }{3} \)^{-\frac{1}{2}}.
\]

In \cite{MiaoXuZhao2013, XuZhao2020}, the authors classified the dynamics of radial $H^1$ solutions to \eqref{NLS} with $E(u_{0})<E^{c}(W)$.  They proved the following (see also \cite{ChengMiaoZhao2016, MiaoZhaoZheng2017, Cheng2020, Luo2022, AkaIbraKikuNawa2013, AkaIbraKikuNawa2021, MAKIWA}):

\begin{theorem}[Sub-threshold dynamics, \cite{XuZhao2020, MiaoXuZhao2013}]\label{Th1}
Let $u_{0}\in H^{1}(\R^{3})$ be radial, and let $u$ be the corresponding maximal-lifespan solution to \eqref{NLS} with $u|_{t=0}=u_{0}$.
\begin{enumerate}[label=\rm{(\roman*)}]
\item If $u_{0}$ obeys
\begin{equation}\label{Thres00}
E(u_{0})<E^{c}(W) \quad \text{and} \quad \|\nabla u_{0}\|_{L^{2}}<\|\nabla W\|_{L^{2}},
\end{equation}
then $u$ is global, $u\in L^{10}_{t, x}(\R\times\R^{3})$, and scatters as $t\to\pm\infty$.
\item If $u_{0}$ obeys
\begin{equation}\label{BlowC00}
E(u_{0})<E^{c}(W) \quad \text{and} \quad \|\nabla u_{0}\|_{L^{2}}>\|\nabla W\|_{L^{2}},
\end{equation}
then $u$ blows up in both time directions in finite time.
\end{enumerate}
\end{theorem}
Here \emph{scattering} as $t\to\pm\infty$ refers to the statement that
\[
\lim_{t\to\pm\infty}\|u(t)-e^{it\Delta}u_\pm\|_{H^1}=0\qtq{for some}u_\pm\in H^1.
\]

In this paper, we show that this dichotomy in Theorem~\ref{Th1} continues to hold for radial $H^1$ solutions at the energy threshold $E(u_0)=E^c(W)$.

\begin{theorem}[Threshold dynamics]\label{Th2} 
Let $u_{0}\in H^{1}(\R^{3})$ be radial and let $u$ be the corresponding maximal-lifespan solution to \eqref{NLS} with $u|_{t=0}=u_0$.
\begin{enumerate}[label=\rm{(\roman*)}]
\item If $u_{0}$ obeys
\begin{equation}\label{Thres}
E(u_{0})=E^{c}(W) \quad \text{and} \quad \|\nabla u_{0}\|_{L^{2}}<\|\nabla W\|_{L^{2}},
\end{equation}
then $u$ is global, $u\in L_{t,x}^{10}(\R\times\R^3)$, and scatters as $t\to\pm\infty$.
\item If $u_{0}$ obeys
\begin{equation}\label{BlowC}
E(u_{0})=E^{c}(W) \quad \text{and} \quad \|\nabla u_{0}\|^{2}_{L^{2}}>\|\nabla W\|^{2}_{L^{2}},
\end{equation}
then $u$ blows up in both time directions in finite time.
\end{enumerate}
\end{theorem}

Generally speaking, we regard \eqref{NLS} as a perturbation of the energy-critical NLS \eqref{ECgNLS}.  For this latter model, one also has a sub-threshold scattering/blowup dichotomy for radial $\dot H^1$-data completely analogous to Theorem~\ref{Th1} \cite{KenigMerle2006}.  As for the threshold dynamics, the work \cite{DuyckaMerle2009} classified all possible behaviors for radial $\dot H^1$-data with energy equal to that of the ground state.  In addition to the possibility of scattering and blowup, the authors constructed two special solutions $W^\pm$, which are `heteroclinic orbits' connecting the ground state to blowup/scattering.  

Comparing this result with Theorem~\ref{Th2}, one sees that the heteroclinic orbits are now missing.  This phenomenon has been previously observed in previous works on focusing NLS with repulsive-type perturbations (see e.g. \cite{DuyLanRou2022} for the focusing cubic NLS outside of a convex obstacle and  \cite{ArdilaInui2022, MiaMurphyZheng2021} for the focusing NLS with repulsive potentials ).  As noted in \cite{DuyckaMerle2009}, however, the heteroclinic orbits are also absent in the setting of \eqref{ECgNLS} if one considers $H^1$-solutions. This is due to the fact that in three space dimensions the ground state $W$, and hence the solutions $W^\pm$, fail to belong to $L^2$.  In particular, it is not immediately clear whether Theorem~\ref{Th2} provides a complete picture of the threshold dynamics, or whether working with a larger class of initial data  may reveal additional solution behaviors.  

We contend that in the setting of \eqref{Thres}, one will only observe scattering, while in the setting of \eqref{BlowC} it is conceivable that additional behaviors are possible.  We plan to pursue these questions in future work.  For some additional results related to questions of this type, we further refer the reader to \cite{CarlesSparber, ArdilaMurphy, Murphy, KMV}, which consider the classification of threshold behaviors for the cubic-quintic NLS with defocusing quintic term and focusing cubic term.

In the rest of this introduction, we outline the structure of the paper and the arguments used to prove Theorem~\ref{Th2}. 

The main technical ingredients needed to establish Theorem~\ref{Th2} are localized virial arguments (see Section~\ref{S:preli}), concentration-compactness (see Section~\ref{S:preli} and Section~\ref{S:Compactness}), and modulation analysis around the ground state (see Section~\ref{S:Modulation}). In addition to the standard localized virial estimates, we will utilize a version of the virial identity that takes into account the fact that the solution is close to the ground state.  The radial assumption allows us to ignore the possibility of a moving spatial center in the concentration-compactness arguments (as well as in the modulation analysis).  In terms of the technical details, we rely on the radial Sobolev embedding to control some error terms arising in the localized virial arguments.  As a radial assumption is used to establish the sharp scattering/blowup results of \cite{KenigMerle2006, DuyckaMerle2009, XuZhao2020, MiaoXuZhao2013}, it is natural to work under this assumption in the present work, as well. 

The proof of blowup in the case \eqref{BlowC} is based on the standard virial argument together with some modulation analysis.  We remark that the conservation of mass is used in an essential way in the proof. In fact, one can observe that if the coefficient
appearing in front of the cubic nonlinearity  is nonnegative,  then it  plays no role in this argument, and in particular could be taken to be equal to zero.  In this case, we reproduce the fact (already observed in \cite{DuyckaMerle2009}) that $H^1$ threshold solutions to \eqref{ECgNLS} obeying \eqref{BlowC} blow up in finite time.  In particular, one does not encounter the special solution $W^+$ in the $H^1$ setting.  For the details, see Section~\ref{Sec22}.

The proof of scattering in the case \eqref{Thres} is based on the concentration-compactness approach.  In particular, we prove that if a solution as in Theorem~\ref{Th2}(i) fails to scatter, then it satisfies a compactness property in $\dot H^1$; see Section~\ref{S:Compactness}.  We then rule out the possibility that such a solution can exist.  We consider separately the possibility of finite-time blowup and infinite-time blowup and obtain a contradiction in either case. 

In the case of finite-time blowup, we find that the solution must move to arbitrarily small spatial scales as one approaches the finite blowup time.  In this case, however, we can prove that the (conserved) mass of the solution is identically zero, which leads to a contradiction.  See Section~\ref{S:FTB}.

In the case of infinite-time blowup, the argument proceeds roughly as follows. Using the standard localized virial argument, we can find a sequence of times along which the solution approaches the orbit of the ground state.  This is made precise using the quantity 
\[
\delta(t) = \bigl| \|\nabla W\|_{L^2}^2 - \|\nabla u(t)\|_{L^2}^2\bigr|. 
\]
We can further show that if $\delta(t_n)\to 0$ for some $t_n\to \infty$, then the solution must move to small spatial scales.  However, the modulation analysis shows that the variation in spatial scale may be controlled by the integral of $\delta(t)$, while the refined virial estimate essentially shows that we may control such time integrals by the value of $\delta(\cdot)$ at the endpoints.  It follows that the solution cannot in fact move to arbitrarily small scales, and hence we obtain a contradiction.  See Section~\ref{S:ITB}. 
%

%
%
%

\subsection*{Notation.} Given positive quantities $A$ and $B$, we use $A\lesssim B$ or $B\gtrsim A$ to denote $A\leq CB$ for some positive constant $C>0$. We also use $A\sim B$ to indicate  $A \lesssim B \lesssim A$. 

For $u:I\times \R^{3}\rightarrow \C$, $I\subset \R$, we write
\[ \|  u \|_{L_{t}^{q}L^{r}_{x}(I\times \R^{3})}=\|  \|u(t) \|_{L^{r}_{x}(\R^{3})}  \|_{L^{q}_{t}(I)}
 \]
with $1\leq q\leq r\leq\infty$. We let $\<\nabla\>=(1-\Delta)^{1/2}$ and we define the Sobolev norms 
\[\|u\|_{H^{s,r}(\R^{3})}:=\|\<\nabla\>^{s} u\|_{L^{r}_{x}(\R^{3})}.\]
Moreover, for $f\in H^{1}(\R^{3})$ we denote
\[
\delta(f):=|\|\nabla W\|_{{L}^{2}}^{2}-\|\nabla f\|_{{L}^{2}}^{2}|.
\]

\subsection*{Acknowledgements} J. M. was supported by NSF grant DMS-2137217.   J. Z. was supported by National key R\&D program of China: 2021YFA1002500, and NSFC grant No.12271051.

\section{Preliminaries}\label{S:preli}
We call $(q,r)$ admissible if $2\leq q,r\leq\infty$ and $\frac{2}{q}+\frac{3}{r}=\frac{3}{2}$. We define $S^{0}(I\times \R^{3})$ via the norm
\[
\| u  \|_{S^{0}(I\times \R^{3})}:= \sup\left\{ \|  u \|_{L_{t}^{q}L^{r}_{x}(I\times \R^{3})}: \quad \text{$(q,r)$ is admissible}\right\}.
\]
The standard Strichartz estimates in three dimensions are stated as follows: 
\begin{lemma}[Strichartz estimates, \cite{GinibreVelo, KeelTao, Strichartz}] 
Let $(q,r)$ be an admissible pair. Then the solution $u$ to $(i\partial_{t}+\Delta) u=F$ for a function $F$ with initial data $u_{0}$
obeys
\[\| u  \|_{L_{t}^{q}L^{r}_{x}(I\times \R^{3})} \lesssim \| u_{0}\|_{L^{2}_{x}(\R^{3})}+ \| F\|_{L_{t}^{\tilde{q}'}L^{\tilde{r}'}_{x}(I\times \R^{3})}, \]
where $2\leq \tilde{q},\tilde{r}\leq\infty$ with $\frac{2}{\tilde{q}}+\frac{3}{\tilde{r}}=\frac{3}{2}$ and for some interval $I\subset \R$. 
\end{lemma}

We will also need some Littlewood--Paley theory. We let $\psi\in C_c^\infty(\R^3)$ be nonnegative and radial, with $\psi(x)=1$ if $|x|\leq 1$  and $\psi(x)=0$ if $|x|\geq \frac{11}{10}$. For $N\in 2^{\mathbb{N}}$, we define the Littlewood-Paley projections as Fourier multiplier operators, namely:  
\[
\widehat{P_{\leq N}f}(\xi):=\psi(\tfrac{\xi}{N})\hat{f}(\xi), \quad \widehat{P_{> N}f}(\xi):=\left[1-\psi(\tfrac{\xi}{N})\right]\hat{f}(\xi),
\]
and
\[
\widehat{P_{N}f}(\xi):=\left[\psi(\tfrac{\xi}{N})-\psi(\tfrac{2\xi}{N})\right]\hat{f}(\xi).
\]
These operators are bounded on Sobolev spaces and obey the following standard estimates.
\begin{lemma}[Bernstein inequalities]\label{BIN}
Fix $1\leq p\leq q\leq \infty$ and $s\geq0$. For $f: \R^{3}\to \C$, we have
\begin{align*}
\| P_{N}f   \|_{L^{q}(\R^{3})}&\lesssim N^{\frac{3}{p}-\frac{3}{q}}\| P_{N}f   \|_{L^{p}(\R^{3})},\\
\| P_{> N}f   \|_{L^{p}(\R^{3})}&\lesssim N^{-s}\| |\nabla|^{s} P_{\geq N}f   \|_{L^{p}(\R^{3})}.
\end{align*}
\end{lemma}

\subsection{Variational analysis} We recall some well-known properties of the ground state $W(x)=(1+\frac{1}{3}|x|^{2})^{-\frac{1}{2}}$, which belongs to $\dot H^1(\R^3)$ but not $L^2(\R^3)$.

The sharp Sobolev inequality takes the form
\begin{equation}\label{GI}
\|f\|_{L^{6}}\leq C_{GN}\|\nabla f\|_{L^{2}},
\end{equation}
with 
\begin{equation}
\label{C_GN}
C_{GN}=\tfrac{\|W\|_{L^{6}}}{\|\nabla W\|_{L^{2}}}.
\end{equation}
We have the following variational characterization of the ground state: 

If $f$ is a nonzero function satisfying $\|f\|_{L^{6}}= C_{GN}\|\nabla f\|_{L^{2}}$, then 
\[
f(x)=e^{i\theta}\lambda^{\frac{1}{2}}W(\lambda(x-x_{0}))
\]
for some $\theta\in\R$, $\lambda>0$ and $x_{0}\in \R^{3}$.

The ground state $W$ also satisfies the following Pohozaev identities:
\begin{equation}\label{PoQ}
\|\nabla W\|^{2}_{L^{2}}=\| W\|^{6}_{L^{6}} \qtq{and} E^{c}(W)=\tfrac{1}{3}\|\nabla W\|^{2}_{L^{2}}.
\end{equation}

We also remark that for $f\in\dot H^1$ satisfying $\|\nabla  f\|_{L^{2}}\leq \|\nabla W\|_{L^{2}}$, we have
\begin{align}\label{IneW}
	\frac{\|\nabla  f\|^{2}_{L^{2}}}{\|\nabla W\|^{2}_{L^{2}}}
\leq \frac{E^{c}(f)}{E^{c}(W)}
\end{align}
(see \cite[Claim 2.6]{DuyckaMerle2009}).

\begin{lemma}\label{GlobalS}
Assume that $u_{0}\in H^{1}(\R^{3})$ satisfies
\begin{equation}\label{Condition11}
E(u_{0})=E^{c}(W)
\quad \text{and}\quad 
\|\nabla u_{0}\|_{L^{2}}<\|\nabla W\|_{L^{2}}.
\end{equation}
Then the solution $u$ to \eqref{NLS} with $u|_{t=0}=u_0$ satisfies 
\begin{equation}\label{PositiveP}
\|\nabla u(t)\|_{L^{2}}<\|\nabla W\|_{L^{2}}
\end{equation}
throughout its maximal lifespan.  Similarly, if $\|\nabla u_0\|_{L^2}>\|\nabla W\|_{L^2}$, then $\|\nabla u(t)\|_{L^2}>\|\nabla W\|_{L^2}$ throughout the maximal lifespan of $u$. 
\end{lemma}

\begin{proof} Suppose that there exists $t_{0}$ such  that $\|\nabla u(t_{0})\|^{2}_{L^{2}}=\|\nabla W\|^{2}_{L^{2}}$. By \eqref{IneW} and \eqref{Condition11}, we obtain
\[
1=\frac{\|\nabla  u(t_{0})\|^{2}_{L^{2}}}{\|\nabla W\|^{2}_{L^{2}}}\leq \frac{E^{c}(u(t_{0}))}{E^{c}(W)}
<\frac{E(u(t_{0}))}{E^{c}(W)}=1,
\]
which is a contradiction.  A similar argument treats the remaining case. \end{proof}

%

\subsection{Virial identities}\label{S:virial} Let $\phi$ be a smooth radial function satisfying
\[
\phi(x)=
\begin{cases}
|x|^{2},& \quad |x|\leq 1\\
0,& \quad |x|\geq 2,
\end{cases}
\quad \text{with}\quad 
|\partial^{\alpha}\phi(x)|\lesssim |x|^{2-|\alpha|}
\]
for all multiindices $\alpha$.  Given $R>1$ we define
\begin{equation}\label{WRR}
w_{R}(x)=R^{2}\phi\(\tfrac{x}{R}\)
\end{equation}
and introduce the localized virial functional
\begin{equation}\label{I_R}
I_{R}[u]=2\IM\int_{\R^{3}} \nabla w_{R} \cdot\nabla u \,\overline{u} \,dx.
\end{equation}

We have the following virial identity (see e.g., \cite[Lemma 2.5]{MiaoXuZhao2013}).

\begin{lemma}\label{VirialIden}
Let $R\in [1, \infty)$. Assume that $u$ solves \eqref{NLS}. Then we have
\begin{equation}\label{LocalVirial}
\tfrac{d}{d t}I_{R}[u]=F_{R}[u],
\end{equation}
where
\begin{align*}
F_{R}[u]&:=\int_{\R^{3}}\big\{(- \Delta \Delta w_{R})|u|^{2}+
\Delta[w_{R}]|u|^{4}+4\RE \overline{u_{j}} u_{k} \partial_{jk}[w_{R}]\big\}dx\\
&\quad -\tfrac{4}{3}\int_{\R^{3}}\Delta[w_{R}]|u|^{6}dx\\
&=:F^{c}_{R}[u]+\int_{\R^{3}}\Delta[w_{R}]|u|^{4}dx.
\end{align*}
\end{lemma}

When $R=\infty$, we denote $F^{c}_{\infty}[u]=8G[u]$,
where
\[
G[f]:=\|\nabla f\|^{2}_{L_{x}^{2}}-\|f\|^{6}_{L_{x}^{6}}.
\]

\begin{lemma}\label{Virialzero}
Fix $R\in [1, \infty]$, $\theta\in \R$ and $\lambda >0$. Then we have
\[
F^{c}_{R}[e^{i\theta}\lambda^{\frac{1}{2}}W(\lambda\cdot)]=0.
\]
\end{lemma}
\begin{proof}
If $R=\infty$, then by a change of variable and \eqref{PoQ} we get
\[
F^c_{\infty}[e^{i\theta}\lambda^{\frac{1}{2}}W(\lambda\cdot)]
=8G[e^{i\theta}\lambda^{\frac{1}{2}}W(\lambda\cdot)]
=8G[W]=0.
\]
Now assume $R\in [1, \infty)$. Fix $\theta\in \R$ and $\lambda >0$. Since 
\[
u(t,x)=e^{i\theta}\lambda^{\frac{1}{2}}W(\lambda x)
\]
is a solution to \eqref{ECgNLS} and $I_{R}[e^{i\theta}\lambda^{\frac{1}{2}}W(\lambda\cdot)]=0$, Lemma~\ref{VirialIden} implies
\[
F^{c}_{R}[e^{i\theta}\lambda^{\frac{1}{2}}W(\lambda\cdot)]=0.
\]
\end{proof}

Combining Lemmas~\ref{VirialIden} and \ref{Virialzero} yields the following, which we will use to incorporate the modulation analysis into the virial analysis. 

\begin{lemma}\label{VirialModulate}
Consider $R\in [1, \infty]$, $\chi: I\to \R$, $\theta: I\to \R$ and
$\lambda: I\to \R$. If $u$ is a solution  to \eqref{NLS}, then we have
\begin{align}
\tfrac{d}{d t}I_{R}[u]=&F^{c}_{\infty}[u(t)]\nonumber\\
&+F_{R}[u(t)]-F^{c}_{\infty}[u(t)] \label{Modu11}\\ 
&-\chi(t)\big\{F^{c}_{R}[e^{i\theta(t)}\lambda(t)^{\frac{1}{2}}W(\lambda(t)\cdot)]-F^{c}_{\infty}[e^{i\theta(t)}\lambda(t)^{\frac{1}{2}}W(\lambda(t)\cdot)]\big\}.\label{Modu22}
\end{align}
\end{lemma}

\subsection{Well-posedness, stability, and concentration-compactness}

We first recall the following result from \cite[Proposition 3.2]{TaoVisanZhang2007}. 

\begin{proposition}[Local well-posedness]\label{LWP}
For any $u_{0}\in H^{1}(\R^{3})$, there exists a unique solution $u\in C(I, H^{1}(\R^{3}))$ to \eqref{NLS} on some interval $I=(-T_{{min}}, T_{{max}})\ni 0$ such that:
\begin{enumerate}[label=\rm{(\roman*)}]
\item  The solution satisfies the conservation of energy and mass 
\begin{equation*}
E(u(t))=E(u_{0})\quad  \text{and} \quad M(u(t))=M(u_{0}) 
\end{equation*}
for all $t\in I$, where
\[
M(u)=\tfrac{1}{2}\int_{\R^{3}}|u|^{2}dx.
\]
\item $u\in L_{t}^{q}H_{x}^{1,r}(K\times \R^{3})$  for every  compact time interval $K\subset I$ and for any admissible pair $(q,r)$.
\item  If $T_{{max}}$ is finite, then 
\[
\text{ either 
$\lim_{t\rightarrow T_{{max}}}\|u(t)\|_{\dot{H}^{1}_{x}}=\infty$ \quad or\quad  $\|\nabla u\|_{S^{0}((0,T_{{max}})\times \R^{3})}=\infty$}.
\]
 An analogous statement holds when $T_{{min}}$ is finite.
\end{enumerate}
\end{proposition}

\begin{remark}[Persistence of regularity]\label{Rpr}
Standard arguments also yield the following: if $u: \R\times \R^{3}\rightarrow \C$ a solution to \eqref{NLS} with $\|u\|_{L^{10}_{t,x}(\R\times\R^{3})}\leq L<\infty$, then
\begin{equation}\label{S1}
\||\nabla|^{s}u\|_{S^{0}(\R\times\R^{3})}\leq C(L, M(u(t_{0})))\||\nabla|^{s}u(t_{0})\|_{L^{2}_{x}}, \quad s\in\{0,1\},
\end{equation}
for any $t_{0}\in \R$.   This is proven by combining Strichartz estimates and a standard continuity argument.  The relevant nonlinear estimates are as follows: 
\begin{equation}\label{E12}
\begin{split}
\||\nabla|^{s} |u|^{4}u  \|_{L_{t}^{\frac{10}{9}}L^{\frac{30}{17}}_{x}}\lesssim
\| u  \|^{4}_{L_{t,x}^{10}}\||\nabla|^{s} u  \|_{L_{t}^{2}L^{6}_{x}},
\end{split}
\end{equation}
and 
\begin{equation}\label{E33}
\begin{split}
\||\nabla|^{s} |u|^{2}u  \|_{L_{t}^{\frac{5}{3}}L^{\frac{30}{23}}_{x}}\lesssim &
\| u  \|_{L_{t}^{\infty}L_{x}^{2}}\| |u||\nabla|^{s} u  \|_{L_{t}^{\frac{5}{3}}L^{\frac{15}{4}}_{x}}\\
\lesssim& \| u  \|_{L_{t}^{\infty}L^{2}_{x}}\| u  \|_{L_{t,x}^{10}}\||\nabla|^{s} u  \|_{L_{t}^{2}L^{6}_{x}}.
\end{split}
\end{equation}
for $s\in\{0,1\}$. 
\end{remark}

This result leads to the following sufficient condition for scattering in $H^1$:

\begin{proposition}\label{ScatterCondi}
Let $u_{0}\in H^{1}(\R^{3})$, and let $u$ be the solution to \eqref{NLS} with $u|_{t=0}=u_{0}$. If $u$ is global and $u\in L_{t,x}^{10}(\R\times\R^3)$, then $u$ scatters in $H^{1}(\R^{3})$ as $t\to \pm \infty$.
\end{proposition}

We next record a stability result for \eqref{NLS} (see e.g. \cite[Proposition 6.3]{KillipOhPoVi2017}):

\begin{lemma}[Stability result]\label{stabi}
Let $t_0\in I\subset\R$. Suppose $\tilde{u}:I\times\R^3\to\C$ solves
\[ 
(i\partial_{t}+\Delta) \tilde{u} =|\tilde{u}|^{2}\tilde{u}-|\tilde{u}|^{4}\tilde{u}+e,\quad  \tilde{u}(t_{0})=\tilde{u}_{0}
\]
for some $e:I\times\R^{3}\rightarrow \C$. Assume also that
\[
\| \tilde{u}  \|_{L_{t}^{\infty}H^{1}_{x}(I\times\R^{3})}+\|\tilde u\|_{L_{t,x}^{10}(I\times\R^3)}\lesssim A.
\]
Let $u_{0}\in H^{1}(\R^{3})$, and suppose $\|u_{0}\|_{L^{2}}\leq M$ for some $M>0$.  Then there exists $\epsilon_{0}=\epsilon_0(A,M)>0$ such that if $0<\epsilon<\epsilon_{0}$ and
\[
\|u_{0}-\tilde{u}_{0} \|_{\dot{H}^{1}_{x}}\leq \epsilon\qtq{and} \|\nabla e  \|_{L_{t,x}^{\frac{10}{7}}(I\times\R^{3})}\leq \epsilon,
\]
then there exists a unique solution $u: I\times \R^{3}\to \C$ to \eqref{NLS} with the specified initial data $u_{0}$ at the time $t=t_{0}$
satisfying
\begin{equation}\label{SCN55}
\|\nabla (u-\tilde{u})\|_{S^{0}(I)}\lesssim_{A,M}\epsilon.
\end{equation}
\end{lemma}

We next import the following linear profile decomposition (see
 \cite[Theorem 7.5]{KillipOhPoVi2017} for more details).  Due to the radial assumption, all of the spatial translation parameters may be set to zero. 
  
\begin{theorem}[Linear profile decomposition]\label{Profi}
Let $\left\{f_{n}\right\}_{n\in \N}$ be a bounded sequence of radial functions in $H^{1}(\R^{3})$. Passing to a subsequence, there exist  $J^{\ast}\in \left\{0,1,2,\ldots\right\}\cup\left\{\infty\right\}$, nonzero profiles $\{\phi^{j}\}^{J^{\ast}}_{j=1}\subset \dot{H}_{x}^{1}(\R^{3})\setminus\left\{0\right\}$, and parameters $\{(\lambda^{j}_{n}, t^{j}_{n})\}^{J^{\ast}}_{j=1}\subset [1, +\infty)\times \R\times\R^{3}$ satisfying for any fixed $j$,
\begin{itemize}
\item $\lambda^{j}_{n}\equiv 1$ or $\lambda^{j}_{n}\rightarrow +\infty$, and $t^{j}_{n}\equiv 0$ or $t^{j}_{n}\rightarrow\pm\infty$,
	\item If $\lambda^{j}_{n}\equiv 1$ then $\phi^{j}\in L^{2}(\R^{3})$.
\end{itemize}
Also we can write
\begin{equation}\label{Dcom}
f_{n}=\sum^{J}_{j=1}\phi_{n}^{j}+R_n^J
\end{equation}
for each finite $1\leq J\leq J^{\ast}$, with
\begin{equation}\label{fucti}
\phi_{n}^{j}(x):=
\begin{cases} 
[e^{it^{j}_{n}\Delta}\phi^{j}](x), \quad \mbox{if $\lambda^{j}_{n}\equiv 1$},\\
(\lambda^{j}_{n})^{\frac{1}{2}}[e^{it^{j}_{n}\Delta} P_{\geq (\lambda^{j}_{n})^{-\theta}}\phi^{j}]\( \lambda^{j}_{n}x\),
\quad \mbox{if $\lambda^{j}_{n}\rightarrow +\infty$},
\end{cases}
\end{equation}
for some $0<\theta<1$. Furthermore,  we have the following properties:
\begin{itemize}
\item Vanishing of the remainder: for every $\eps>0$, there exists $J_0=J_0(\eps)$ such that 
\begin{equation}\label{Sr} 
\limsup_{n\rightarrow\infty}\|e^{it\Delta}R_n^J\|_{L^{10}_{t,x}(\R\times\R^{3})}<\epsilon
   \end{equation}
   for $J\geq J_0$.
\item Weak convergence of the remainder:
		\begin{equation}\label{Wcp}
e^{-it^{j}_{n}\Delta}[(\lambda^{j}_{n})^{-\frac{1}{2}}R_n^J((\lambda^{j}_{n})^{-1}x)]\rightharpoonup 0\quad \mbox{in}\,\,
\dot{H}^{1}_{x}, \quad \mbox{as $n\rightarrow\infty$.}
     \end{equation}
\item Orthogonality: for all $1\leq j\neq k\leq J^{\ast}$
\begin{equation}\label{Pow}
\lim_{n\rightarrow \infty}\left[ \tfrac{\lambda^{j}_{n}}{\lambda^{k}_{n}}+\tfrac{\lambda^{k}_{n}}{\lambda^{j}_{n}} +
\lambda^{j}_{n}\lambda^{k}_{n}{|t^{j}_{n}(\lambda^{j}_{n})^{-2}-t^{k}_{n}(\lambda^{k}_{n})^{-2}|}\right]=\infty.
\end{equation}		
\item Decoupling: for any $J\in \N$,
\begin{align}\label{PE11}
\| f_{n}  \|^{p}_{L^{p}_{x}}&=\sum^{J}_{j=1}\| \phi^{j}_{n}  \|^{p}_{L^{p}_{x}}+\| R_n^J  \|^{p}_{L^{p}_{x}}+o_{n}(1),\quad p\in\{2,4,6\}\\
\| f_{n}  \|^{2}_{\dot{H}^{1}_{x}}&=\sum^{J}_{j=1}\| \phi^{j}_{n}  \|^{2}_{\dot{H}^{1}_{x}}+
\| R_n^J  \|^{2}_{\dot{H}^{1}_{x}}+o_{n}(1).
\end{align}
\end{itemize}
\end{theorem}

As an immediate consequence of the Theorem~\ref{Profi} we have the following result.
\begin{lemma}\label{Lde}
Under the hypotheses of Theorem \ref{Profi}, we have
\begin{equation}\label{PE}
E(f_{n})=\sum^{J}_{j=1}E(\phi^{j}_{n} )+E(R_n^J)+o_{n}(1),
\end{equation}
for any $J\in \N$.
\end{lemma}

\subsection{Approximation by the quintic NLS}  An important step in the proof of Theorem~\ref{Th2} is the approximation of the cubic-quintic NLS by the pure quintic NLS in the small-scale limit.  A key ingredient in this approximation argument is the following scattering result for the quintic NLS appearing in \cite{KenigMerle2006}:

\begin{theorem}\label{WGPQN}
Let $u_{0}\in \dot{H}^{1}(\R^{3})$ be radial. Assume 
\[
E^{c}(u_{0})<E^{c}(W)\qtq{and}\|\nabla u_{0}\|_{L^{2}}<\|\nabla W\|_{L^{2}}.
\]
Then there exists a unique global solution $u\in C(\R, \dot{H}_{x}^{1}(\R^{3}))$ of  the energy-critical NLS \eqref{ECgNLS} with $u|_{t=0}=u_{0}$, which satisfies
\[
\| u  \|_{L^{10}_{t,x}(\R\times\R^{3})} \lesssim 1.
\]
In addition, the solution scatters in both time directions in $\dot{H}^{1}(\R^{3})$.
\end{theorem}

Building on this result, we will establish the following approximation result, which follows along the same lines as \cite[Proposition 8.3]{KillipOhPoVi2017}.

\begin{lemma}[Embedding nonlinear profiles]\label{P:embedding}
Fix $\theta\in(0,1)$.  Let $\{t_n\}_{n\in \N}$ satisfy $t_n\equiv 0$ or $t_n\to\pm\infty$, and let $\{\lambda_n\}_{n\in \N}\subset [1, \infty)$ satisfy 
$\lambda_{n} \to +\infty$ as $n\to+\infty$.
Suppose $\phi\in \dot{H}^{1}(\R^{3})$ is radial and obeys 
\begin{align}\label{PriC}
&E^{c}(\phi)<E^{c}(W) \qtq{and} \|\nabla \phi\|^{2}_{L^{2}}< \|\nabla W\|^{2}_{L^{2}}\qtq{if} t_{n}\equiv0,\\
&\tfrac{1}{2}\|\nabla \phi\|^{2}_{{L}^{2}}<E^{c}(W) \qtq{if} t_{n}\to \pm \infty.\label{seC}
\end{align}
For $n$ sufficiently large, there exists a global solution $v_n$ to \eqref{NLS} with 
\[
v_n(0,x)=\phi_n(x):=(\lambda_{n})^{\frac{1}{2}}[e^{it_{n}\Delta} P_{\geq (\lambda_{n})^{-\theta}}\phi]( \lambda_{n}x)
\]
that satisfies
\begin{align}\label{ImporB}
\|\nabla v_n\|_{S^{0}(\R\times \R^{3})}\lesssim 1.
\end{align}
Moreover, for every $\epsilon>0$ there exist $N=N(\epsilon)$ and $\chi_\eps,\psi_\eps\in C_c^\infty(\R\times\R^3)$ such that
\begin{align}\label{Small11}
\| v_{n}-\lambda^{1/2}_{n}\chi_{\epsilon}(\lambda^{2}_{n}{t}+t_{n},\lambda_{n}x)  \|_{L^{10}_{t,x}}&<\epsilon,\\
	\label{Small22}
\| \nabla v_{n}-\lambda^{3/2}_{n}\psi_{\epsilon}(\lambda^{2}_{n}{t}+t_{n}, \lambda_{n}x)  \|_{L^{\frac{10}{3}}_{t,x}}&<\epsilon,
\end{align}
for all $n\geq N$.
\end{lemma} 

\begin{proof}[Sketch of the proof] If $t_{n}\equiv0$, we let $w_{n}$ and $w$ be the solutions to \eqref{ECgNLS} with initial data $w_{n}(0)=P_{\geq (\lambda_{n})^{-\theta}}\phi$ and $w(0)=\phi$, respectively. If $t_{n} \to \pm \infty$, we let $w_{n}$ and $w$ be the solutions to \eqref{ECgNLS} that scatter in $\dot{H}^{1}$ to $e^{it \Delta}P_{\geq (\lambda_{n})^{-\theta}}\phi$ and $e^{it \Delta}\phi$, respectively, as $t_{n} \to \pm \infty$. Note that $P_{\geq (\lambda_{n})^{-\theta}}\phi$ is radial and
\begin{align}\label{Ap11}
\| P_{\geq (\lambda_{n})^{-\theta}}\phi-\phi \|_{\dot{H}^{1}_{x}}\rightarrow 0 \quad \text{as $n\rightarrow \infty$}.
\end{align}
By conditions \eqref{PriC} and \eqref{seC} and Theorem~\ref{WGPQN}, we deduce that in either case, for $n$ sufficiently large, $w_{n}$ and $w$ are global and obey
\begin{align}\label{BoStr}
\|\nabla w_{n}   \|_{S^{0}(\R\times \R^{3})}+\|\nabla w   \|_{S^{0}(\R\times \R^{3})} \lesssim 1.
\end{align}
On the other hand, Lemma~\ref{BIN} implies
\[
\| P_{\geq (\lambda_{n})^{-\theta}}\phi\|_{{L}^{2}}\lesssim\lambda^{\theta}_{n}\| \nabla \phi\|_{{L}^{2}},
\]
thus, by persistence of regularity (see e.g. Remark~\ref{Rpr}), we deduce 
\begin{align}\label{wpersi11}
	\|w_{n}\|_{S^{0}(\R\times \R^{3})}\leq C(\|\phi\|_{\dot{H}^{1}_{x}})\lambda^{\theta}_{n}.
\end{align}

We now define
\[
\tilde{u}_{n}(t,x)=\lambda^{\frac{1}{2}}_{n}w_{n}(\lambda^{2}_{n}t+t_{n}, \lambda_{n}x),
\]
which also solve  \eqref{ECgNLS}.  We will show that the $\tilde{u}_{n}$ are approximate solutions to \eqref{NLS} that asymptotically match $\phi_{n}$ at $t=0$. Indeed, first notice that
\[
\|\nabla \tilde{u}_{n}   \|_{S^{0}(\R\times \R^{3})}=\|\nabla w_{n}   \|_{S^{0}(\R\times \R^{3})}
\lesssim 1
\]
and 
\[
\| \tilde{u}_{n}(0)-\phi_{n} \|_{\dot{H}^{1}_{x}}
=\| w_{n}(t_{n})-e^{it_{n}\Delta}P_{\geq (\lambda_{n})^{-\theta}}\phi_{n} \|_{\dot{H}^{1}_{x}}
\rightarrow 0 \quad \text{as $n\rightarrow \infty$.}
\]
Moreover, we deduce from \eqref{wpersi11} (and the fact that $\lambda_n\to\infty$) that
\[
\|\tilde{u}_{n}\|_{S^{0}(\R\times \R^{3})}
=\lambda^{-1}_{n}\|w_{n}\|_{S^{0}(\R\times \R^{3})}
\leq C(\|\phi\|_{\dot{H}^{1}_{x}})\lambda^{-(1-\theta)}_{n}\to 0
\]
as $n\to \infty$.

We now define the errors $e_{n}:=-|\tilde{u}_{n}|^{2}\tilde{u}_{n}$.  The estimates given above imply that 
\begin{equation}\label{Inter1}
\begin{split}
\|\nabla   e_{n} \|_{L^{\frac{10}{7}}_{t, x}}\lesssim 
\|  \nabla \tilde{u}_{n}\|_{L^{\frac{10}{3}}_{t,x}}\| \tilde{u}_{n}  \|_{L^{10}_{t,x}}\| \tilde{u}_{n}\|_{L^{\frac{10}{3}}_{t,x}}
\lesssim
\lambda^{-(1-\theta)}_{n}\rightarrow0
 \end{split}
\end{equation}
as $n\rightarrow\infty$, with all space-time norms over $\R\times \R^{3}$. Thus, we can apply the stabilty result (Lemma \ref{stabi}) to deduce the existence of  global solutions $v_{n}$ of \eqref{NLS} with  data  $u_{n}(0)=\phi_{n}$ such that \eqref{ImporB} holds.

Finally, the long-time perturbation theory for \eqref{ECgNLS} (see e.g. \cite[Theorem 2.14]{KenigMerle2006}) implies that
\[
\lim_{n\to \infty}\|\nabla (w_{n}-w)\|_{S^{0}(\R\times \R^{3})}=0.
\]
One then obtains \eqref{Small11} and \eqref{Small22} via the arguments in \cite[Proposition~8.3]{KillipOhPoVi2017}. \end{proof}

\section{Modulation analysis}\label{S:Modulation}

In this section we assume that $u:I\times\R^3\to\C$ is a radial solution to \eqref{NLS} with
\begin{equation}\label{22condition}
E(u_{0})=E^{c}(W).
\end{equation}
We write $\delta(t):=\delta(u(t))$, where
\[
\delta(u):=\bigl|\|\nabla W\|_{{L}^{2}}^{2}-\|\nabla u\|_{{L}^{2}}^{2}\bigr|.
\]

For $\delta_{0}>0$ small (to be determined below), we define
\[
I_{0}=\{t\in I:\delta(u(t))<\delta_0\}.
\]

\begin{proposition}\label{Modilation11}
There exist $\delta_{0}>0$ sufficiently small and functions $\theta: I_{0}\to \R$ and $\mu: I_{0}\to \R$ so that writing
\begin{equation}\label{DecomU}
u(t,x)=e^{i \theta(t)}[g(t)+\mu(t)^{\frac{1}{2}}W(\mu(t)x)]
\end{equation}
for $t\in I_0$, we have 
\begin{align}
&\|u(t)\|^{2}_{L_{x}^{4}} +\tfrac{1}{1+\mu(t)}
\lesssim\delta(t)\sim \|g(t)\|_{\dot{H}^{1}},\label{Estimatemodu}\\
&\left|\tfrac{\mu^{\prime}(t)}{\mu(t)}\right|\lesssim \mu^{2}(t)\delta(t).\label{EstimLaD}
\end{align}
\end{proposition}

We begin with the following lemma. 

\begin{lemma}\label{LemmaMod}  
For any $\epsilon>0$, there exists $\delta_{0}=\delta_{0}(\epsilon)>0$ small enough that if $\delta(u(t))<\delta_{0}$, then there exists $(\theta_{0}(t), \mu_{0}(t))\in \R\times [0, \infty)$ so that
\begin{equation}\label{CondiModu}
\| u(t,\cdot)-e^{i\theta_{0}(t)}\mu_{0}(t)^{\frac{1}{2}}W(\mu_{0}(t)\cdot)\|_{\dot{H}^{1}}<\epsilon.
\end{equation}
\end{lemma}
\begin{proof} Suppose instead that there exist $\epsilon>0$ and $\left\{t_{n}\right\}$ so that
\begin{equation}\label{Contra11}
\delta(u(t_{n}))\to 0 \qtq{but}
\inf_{\theta\in \R}\inf_{\mu>0}\|u(t_{n},\cdot)-e^{i\theta}\mu^{\frac{1}{2}}W(\mu\cdot)\|_{\dot{H}^{1}}\geq \epsilon
\end{equation}
for all $n$. We now recall the variational characterization of the ground state:
\begin{align}\label{Vrat}
E^{c}(W)=\inf\left\{E^{c}(f) : \|\nabla f\|_{L^{2}} =\|\nabla W\|_{L^{2}}, \,\,\ f\in \dot{H}^{1}\right\}.
\end{align}
As $\|\nabla u(t_{n})\|^{2}_{L^{2}} \to \|\nabla W\|^{2}_{L^{2}}$ and  $E^{c}(u(t_{n}))\leq E^{c}(W)$  for all $n$ (cf. \eqref{22condition}), we see that $\left\{u(t_{n})\right\}$ is a minimizing sequence for \eqref{Vrat}. Thus there exists $(\theta_{n}, \mu_{n})\in \R^{2}$ such that
$e^{i\theta_{n}}\mu^{-\frac{1}{2}}_{n}u(t_{n},\mu^{-1}_{n}\cdot)\to W$ in $\dot{H}^{1}(\R^{3})$, contradicting \eqref{Contra11}.
\end{proof}


\begin{remark}\label{Yinfinity}
Let $R\geq1$. If $\delta_{0}$ is sufficiently small, we may assume
\begin{equation}\label{Nobound}
\mu_{0}(t)\geq R
\end{equation}
for $t$ in the lifespan of $u$.  Indeed, if \eqref{Nobound} fails, then there exists $\left\{t_{n}\right\}$ so that
\begin{equation}\label{Bounddelta}
\delta(t_{n})\to0\qtq{and}\mu_{0}(t_{n})\leq R\qtq{for all}n\in \N.
\end{equation}
From \eqref{CondiModu} we then deduce that
\begin{align}\label{Cv112}
f_{n}(x):=	e^{-i \theta_{0}(t_{n})}\tfrac{1}{\mu_{0}(t_{n})^{\frac{1}{2}}}u(t_{n}, \mu_{0}(t_{n})^{-1}x)\to W(x) \qtq{in} \dot{H}^{1}(\R^{3}).
\end{align}
By mass conservation and \eqref{Bounddelta}, we also get 
\[
\|f_{n}\|_{L^{2}}\leq \mu_{0}(t_{n})M(u_{0})^{\frac{1}{2}}\lesssim 1
\qtq{for all $n\in \N$.}
\]
But then by \eqref{CondiModu} we obtain  $f_{n}\rightharpoonup W$ in $L^{2}$ as $n\to \infty$, which contradicts $W\notin L^{2}$.
\end{remark}

Arguing as in \cite[Lemma 3.6]{DuyckaMerle2009} (using Lemma~\ref{LemmaMod} and the implicit theorem), we may also obtain the following lemma:

\begin{lemma}\label{ExistenceF}
If  $\delta_{0}>0$ is sufficiently small, then there exist $C^{1}$ functions $\theta: I_{0}\to \R$ and $\mu: I_{0}\to [0, \infty)$  so that
\begin{equation}\label{Taylor}
\| u(t)-e^{i\theta(t)}\mu(t)^{\frac{1}{2}}W(\mu(t)\cdot)\|_{\dot{H}^{1}} \ll 1.
\end{equation}
Writing 
\[
g(t):=g_{1}(t)+i g_{2}(t)=e^{-i\theta(t)}[u(t)-e^{i\theta(t)}\mu(t)^{\frac{1}{2}}W(\mu(t)\cdot)],
\]
and $W_{1}:=\tfrac{1}{2}W+x\cdot\nabla W\in \dot{H}^{1}$, we have
\begin{equation}\label{Ortogonality}
\<\nabla g_{2}(t), \nabla[\mu(t)^{\frac{1}{2}} W(\mu(t)\cdot)] \>
=\<\nabla g_{1}(t), \nabla[\mu(t)^{\frac{1}{2}} W_{1}(\mu(t)\cdot)]\>\equiv 0.
\end{equation}
\end{lemma}

We will need to expand the energy around the ground state.  We claim that 
\begin{equation}\label{Tay22}
E^{c}(u(t))-E^{c}(W)
=\F(\mu^{-\frac{1}{2}}(t)g(t, \mu^{-1}(t)x))+
o(\|g\|^{2}_{\dot{H}^{1}}),
\end{equation}
where $\F$ is the quadratic form on $\dot{H}^{1}$ defined by
\[
\F(h):=\tfrac{1}{2}\<(E^{c})^{\prime\prime}(W)[h], h\>=
\tfrac{1}{2}\int_{\R^{3}}|\nabla h|^{2}dx-\tfrac{1}{2}\int_{\R^{3}}W^{4}[5|h_{1}|^{2}+|h_{2}|^{2}]dx,
\]
with  $h=h_{1}+ih_{2}\in \dot{H}^{1}$.

To prove \eqref{Tay22}, first observe that using \eqref{Taylor}, we may write 
\[
\begin{split}
&E^{c}(u(t))-E^{c}(W)=E^{c}(e^{-i\theta(t)}u(t)) -E^{c}(\mu^{\frac{1}{2}}(t)W(\mu(t)\cdot))\\
&=\< (E^{c})^{\prime}(\mu^{\frac{1}{2}}(t)W(\mu(t)\cdot), g(t) \>
+\tfrac{1}{2}\< (E^{c})^{\prime\prime}(\mu^{\frac{1}{2}}(t)W(\mu(t)\cdot)[g(t)], g(t)\>+
o(\|g\|^{2}_{\dot{H}^{1}}).
\end{split}
\]
As $(E^{c})^{\prime}(\mu^{\frac{1}{2}}(t)W(\mu(t)\cdot)=0$, we obtain \eqref{Tay22}.

The quadratic form  $\F$ can be written
\[
\F(h)=\tfrac{1}{2}\<L_{1}h_{1},h_{1}\>+\tfrac{1}{2}\<L_{2}h_{2}, h_{2} \>, \qtq{where} h=h_{1}+ih_{2}\in\dot{ H}^{1}(\R^{3})
\]
and $L_{1}$ and $L_{2}$ are the bounded operators defined on $\dot{H}^{1}(\R^{3})$ by
\begin{align*}
L_{1}u=-\Delta u- 5W^{4}u,\quad L_{2}v&=-\Delta v-W^{4}v.	
\end{align*}

Writing $H:=\text{span}\left\{W, iW,W+x\cdot \nabla W\right\}$ (viewed as a subspace of $\dot H^1$), we have the following (see \cite[Claim~3.5]{DuyckaMerle2009}):
\begin{lemma}\label{CoerS}
There exists $C>0$ such that for radial $h=h_{1}+ih_{2}\in H^{\bot}$,
\[
\F(h)\geq C \|h\|^{2}_{\dot{H}^{1}}.
\]
\end{lemma}

We can now begin proving the estimates appearing in Proposition~\ref{Modilation11}.

\begin{lemma}\label{BoundI} Let $(\theta(t), \mu(t))$ and $g(t)$ be as in Lemma \ref{ExistenceF}.  Then
\begin{equation}\label{DeltaBound}
\int_{\R^{3}}|u(t,x)|^{4}dx\lesssim \delta^{2}(t)\sim \|g(t)\|^{2}_{\dot{H}^{1}}.
\end{equation}
\end{lemma}
\begin{proof}
 By \eqref{Tay22} we  see that (recall $E(u(t))=E^{c}(W)$)
\begin{equation}\label{Aprox}
0=\F(\mu^{-\frac{1}{2}}(t)g(t, \mu^{-1}(t)\cdot))
+\tfrac{1}{4}\|u(t)\|^{4}_{L^{4}}+
o(\|g\|^{2}_{\dot{H}^{1}}).
\end{equation}
We now decompose $g(t)$ as follows
\begin{equation}\label{Defg}
g(t)=\alpha(t)\mu(t)^{\frac{1}{2}} W(\mu(t)\cdot)+h(t), \quad \text{with} \quad
\alpha(t):=\frac{(\mu^{-\frac{1}{2}}(t)g(t, \mu^{-1}(t)\cdot), W)_{\dot{H}^{1}}}{(W,W)_{\dot{H}^{1}}}.
\end{equation}
Note that $\alpha\in \R$ is chosen to guarantee that
\begin{align}\label{Alfor}
(\mu(t)^{\frac{1}{2}} W(\mu(t)\cdot), h(t))_{\dot{H}^{1}}=0.
\end{align}

Using \eqref{Taylor} and \eqref{Defg} we first observe that 
\begin{equation}\label{BoundR}
|\alpha(t)|\lesssim \|g\|_{\dot{H}^{1}}\ll 1.
\end{equation}
Moreover, by definition of $h$ (cf. \eqref{Alfor}) and  \eqref{Ortogonality} we have $\mu^{-\frac{1}{2}}(t)h(t, \mu^{-1}(t)\cdot)\in H^{\bot}$.
Thus, Lemma \ref{CoerS} implies
\[
\F(\mu^{-\frac{1}{2}}(t)h(t, \mu^{-1}(t)\cdot)) \gtrsim \|h\|^{2}_{\dot{H}^{1}}.
\]
Combining this with \eqref{Aprox} we deduce that
\[
\|h\|^{2}_{\dot{H}^{1}}+\|u\|^{4}_{L^{4}}\lesssim \alpha^{2}+|\alpha\langle L_{1}W, \mu^{-\frac{1}{2}}(t)h(t, \mu^{-1}(t)\cdot) \rangle|+
o(\|g\|^{2}_{\dot{H}^{1}}).
\]
Notice that $L_{2}(W)=0$, so that $L_{1}W=4\Delta W$. Since (recalling \eqref{Alfor})
\[
\langle L_{1}W, \mu^{-\frac{1}{2}}(t)h(t, \mu^{-1}(t)\cdot)\rangle=-4(W, \mu^{-\frac{1}{2}}(t)h(t, \mu^{-1}(t)\cdot))_{\dot{H}^{1}}=0,
\]
the inequality above shows
\[
\|h\|^{2}_{\dot{H}^{1}}+\|u\|^{4}_{L^{4}}\lesssim \alpha^{2}+o(\|h\|^{2}_{\dot{H}^{1}}).
\]
In particular,
\begin{equation}\label{BoundV}
\|h(t)\|^{2}_{\dot{H}^{1}}\lesssim \alpha^{2}(t)
\quad \text{and}\quad 
\|u(t)\|^{4}_{L^{4}}\lesssim \alpha^{2}(t).
\end{equation}

On the other hand, by the orthogonality condition \eqref{Alfor} we obtain
\[
\|g(t)\|^{2}_{\dot{H}^{1}}=\alpha^{2}(t)\| W\|^{2}_{\dot{H}^{1}}+\|h(t)\|^{2}_{\dot{H}^{1}},
\]
which implies by \eqref{BoundR} and \eqref{BoundV} that $\|g\|_{\dot{H}^{1}}\sim |\alpha|$.

Finally,  combining \eqref{Alfor}) and \eqref{BoundV} we obtain
\begin{align*}
\delta(t)&=\bigl|\|W\|^{2}_{\dot{H}^{1}}-\|[(1+\alpha)\mu(t)^{\frac{1}{2}} W(\mu(t)\cdot)+h]\|^{2}_{\dot{H}^{1}}\bigr|\\
	&=2|\alpha| \|W\|^{2}_{\dot{H}^{1}}+\O(\alpha^{2}),
\end{align*}
which shows $\delta\sim |\alpha|$. Combining the estimates above, we have
\[
\|u(t)\|^{4}_{L^{4}}\lesssim \alpha^{2}(t)\sim \delta^{2}(t)\sim \|g(t)\|^{2}_{\dot{H}^{1}}.
\]
\end{proof}

\begin{lemma}\label{BoundEx}
Under the conditions of Lemma~\ref{BoundI}, if $\delta_{0}$ is sufficiently small, then
\begin{equation}\label{EquaEx22}
\tfrac{1}{1+\mu(t)}\lesssim \delta(t) \quad \text{for all $t\in I_{0}$}.
\end{equation}
\end{lemma}
\begin{proof}
By Remark \ref{Yinfinity}, we have that if $\delta_{0}$ is sufficiently small, then
\[
\delta(t)<\delta_{0}\Rightarrow \mu(t)\geq 1.
\]
As $W(x)=(1+\frac{1}{3}|x|^{2})^{-\frac{1}{2}}$ we see that 
\begin{align}\label{EstWl}
	\mu(t)^{1/2}W(\mu(t)x)\geq \tfrac{1}{(1+\mu(t)^{2})^{1/2}} \qtq{for all $|x|\leq 1$ and $t\in I_{0}$.}
\end{align}
On the other hand, by H\"older and Sobolev, we get
\begin{align*}
\|e^{-i\theta(t)}[u(t)-e^{i\theta(t)}\mu(t)^{\frac{1}{2}}W(\mu(t)\cdot)]\|_{L^{4}(B(0,1))}&=\|g(t)\|_{L^{4}(B(0,1))}\\
&\lesssim
\|g(t)\|_{L^{6}}
\lesssim \|g(t)\|_{\dot{H}^{1}}\lesssim \delta(t).
\end{align*}
Thus, by \eqref{DeltaBound} we deduce
\[
\|\mu(t)^{\frac{1}{2}}W(\mu(t)\cdot)]\|_{L^{4}(B(0,1))}\lesssim \delta(t)+\delta(t)^{1/2}
\lesssim \delta(t)^{1/2}.
\]

Now, combining \eqref{EstWl} and the inequality above we get
\[
\tfrac{1}{(1+\mu(t)^{2})^{2}}\lesssim \|\mu(t)^{\frac{1}{2}}W(\mu(t)\cdot)]\|^{4}_{L^{4}(B(0,1))}
\lesssim \delta(t)^{2}.
\]
\end{proof}

\begin{lemma}\label{BoundEx22}
Under the conditions of Lemma \ref{BoundI}, we have the following estimate
\begin{equation}\label{EquaEx}
|\tfrac{\mu^{\prime}(t)}{\mu(t)}|\lesssim \mu^{2}(t)\delta(t) \quad \text{for all $t\in I_{0}$}.
\end{equation}
\end{lemma}
\begin{proof}
We temporarily adopt the following notation
\[
W_{[\mu(t)]}(x)=\mu(t)^{\frac{1}{2}} W(\mu(t)x),
\quad
W_{1, [\mu(t)]}(x)=\mu(t)^{\frac{1}{2}} W_{1}(\mu(t)x),
\]
where $W_{1}=\frac{1}{2}W+x\cdot \nabla W$. In this notation, we have \[
g(t)=e^{-i\theta(t)}[u(t)-e^{i\theta(t)}W_{[\mu(t)]}].
\] 

Using \eqref{NLS}, we may write
\begin{equation}\label{EquationG}
\begin{split}
i\partial_{t}g+\Delta g-\theta^{\prime}g
-\theta^{\prime}W_{[\mu(t)]}+i\tfrac{\mu^{\prime}(t)}{\mu(t)}W_{1, [\mu(t)]}\\
-|u|^{2}e^{-i\theta}u+[f(e^{-i\theta} u)-f(W_{[\mu(t)]})]=0,
\end{split}
\end{equation}
where $f(z)=|z|^{4}z$.  We now observe that for $\varphi$ with $\Delta\varphi\in L^6$, we have 
\begin{align}
|(f(e^{-i\theta}u)&-f(W_{[\mu(t)]}), \varphi)_{\dot{H}^{1}}|\nonumber \\
&=|(f(e^{-i\theta } u)-f(W_{[\mu(t)]}), \Delta \varphi)_{L^{2}}| \label{FIne}\\
& \leq \|e^{-i\theta} u-W_{[\mu(t)]}\|_{L^{6}}[\|e^{-i\theta} u\|^{4}_{L^{6}}+\|W\|^{4}_{L^{6}}]\|\Delta \varphi\|_{L^{6}} \nonumber \\
& \lesssim \delta(t)\|\Delta \varphi\|_{L^{6}}.\nonumber
\end{align}
Moreover, by H\"older, interpolation, conservation of mass, and \eqref{DeltaBound},
\begin{align}
|(|u|^{2}u, \varphi)_{\dot{H}^{1}}| & =|(|u|^{2}u, \Delta\varphi)_{L^{2}}| \label{MassS} \\
& \leq \|u\|^{2}_{L^{4}}  \|u\|_{L^{3}}\|\Delta \varphi\|_{L^{6}} \nonumber \\
&		\lesssim\delta(t)^{1+\frac{2}{3}}\|\Delta \varphi\|_{L^{6}} \lesssim\delta(t)\|\Delta \varphi\|_{L^{6}}.\nonumber
\end{align}
Note also that
\begin{align}\label{Dg1}
&|(\Delta g, W_{[\mu(t)]})_{\dot{H}^{1}}|
	=|( g,\Delta W_{[\mu(t)]})_{\dot{H}^{1}}|\leq \mu^{2}(t)\|g\|_{\dot{H}^{1}}\|\Delta W\|_{\dot{H}^{1}}
	\lesssim \mu^{2}(t)\delta(t),\\\label{Dg2}
&	|(\Delta g, W_{1, [\mu(t)]})_{\dot{H}^{1}}|
\leq \mu^{2}(t)\|g\|_{\dot{H}^{1}}\|\Delta W_{1}\|_{\dot{H}^{1}}
	\lesssim \mu^{2}(t)\delta(t).
\end{align}

We will now prove that 
\begin{align}\label{Ang}
|\theta^{\prime}|\lesssim \mu^{2}(t)\delta(t)+|\tfrac{\mu^{\prime}(t)}{\mu(t)}|\delta(t).
\end{align}
Indeed, by the orthogonality condition \eqref{Ortogonality} we deduce that
\begin{align*}
|(i\partial_{t} g, W_{[\mu(t)]})_{\dot{H}^{1}}|&=|\IM( g, \partial_{t} W_{[\mu(t)]})_{\dot{H}^{1}}| \\
&=\left|\tfrac{\mu^{\prime}(t)}{\mu(t)}\right||\IM( g, W_{[1, \mu(t)]})_{\dot{H}^{1}}|\lesssim\left|\tfrac{\mu^{\prime}(t)}{\mu(t)}\right|\delta(t).
\end{align*}
Thus, taking the $\dot{H}^{1}$ inner product of \eqref{EquationG} with $W_{[\mu(t)]}$ and using \eqref{Dg1}, \eqref{MassS}, and \eqref{FIne} (with $\varphi=W_{[\mu(t)]}$), we obtain
\[
|\theta^{\prime}|\lesssim  |\theta^{\prime}|\delta(t)+ \mu^{2}(t)\delta(t)+|\tfrac{\mu^{\prime}(t)}{\mu(t)}|\delta(t),
\]
where we have used that $\|\Delta W_{[\mu(t)]}\|_{L^{2}}=\mu^{2}(t)\|\Delta W\|_{L^{2}}$.  As $\delta(t)\ll 1$, the estimate above yields \eqref{Ang}.

Next we claim that
\begin{align}\label{Dlam}
\left|\tfrac{\mu^{\prime}(t)}{\mu(t)}\right|\lesssim
	|\theta^{\prime}|\delta(t) + \mu^{2}(t)\delta(t)+|\tfrac{\mu^{\prime}(t)}{\mu(t)}|\delta(t).
\end{align}
Indeed, taking the $\dot{H}^{1}$ inner product of \eqref{EquationG} with $i\,W_{1, [ \mu(t)]}$, the estimates
 \eqref{MassS} and \eqref{FIne} (with $\varphi=i\, W_{1, [ \mu(t)]}$) and \eqref{Dg2} yield
\[
\left|\tfrac{\mu^{\prime}(t)}{\mu(t)}\right|\lesssim
|(i\partial_{t} g, i W_{1, [\mu(t)]})_{\dot{H}^{1}}|+ \mu^{2}(t)\delta(t)
+|\theta^{\prime}|\delta(t).
\]
In addition, by the orthogonality condition \eqref{Ortogonality} we have
\begin{align*}
|(i&\partial_{t} g, i W_{1, [\mu(t)]})_{\dot{H}^{1}}|=|( g, \partial_{t} W_{1, [\mu(t)]})_{\dot{H}^{1}}|\\
&\lesssim|\tfrac{\mu^{\prime}(t)}{\mu(t)}|\|g\|_{\dot{H}^{1}}\|\tfrac{1}{2}W_{1}+x\cdot\nabla W_{1}\|_{\dot{H}^{1}}\lesssim |\tfrac{\mu^{\prime}(t)}{\mu(t)}|\delta(t),
\end{align*}
which implies claim \eqref{Dlam}. Putting together \eqref{Ang} and \eqref{Dlam} yields \eqref{EquaEx}.
\end{proof}

Proposition~\ref{Modilation11} now follows from Lemmas~\ref{ExistenceF}, \ref{BoundI}, \ref{BoundEx}, and \ref{BoundEx22}.

\section{Blowup}\label{Sec22}

In this section we prove Theorem~\ref{Th2}~(ii).  We proceed by proving several lemmas. We recall the definition of $I_R[u]$ from \eqref{I_R}.

\begin{lemma}\label{Lemma11}
Suppose $u$ is a solution to \eqref{NLS} satisfying
\begin{equation}\label{BCon11}
E(u)=E^{c}(W)\qtq{and}
\|\nabla u_{0}\|_{L^{2}}>\|\nabla W\|_{L^{2}}.
\end{equation}
If $u$ is forward-global, then
\begin{align}\label{TBu}
|I_{R}[u(t)]|\lesssim R^{2} \delta(t) \qtq{for all $t\geq 0$.}
\end{align}
\end{lemma}

\begin{proof}  We fix $\delta_1\in(0,\delta_0)$ and obtain estimates separately in the cases $\delta(t)\geq \delta_1$ or $\delta(t)<\delta_1$.  In the first case, we use H\"older and Sobolev embedding to estimate
\begin{align*}
|I_R[u(t)]| &\lesssim R^2\|\nabla u(t)\|_{L^2}^2 \\
& \lesssim R^2\{\delta(t) + \|\nabla W\|_{L^2}^2\} \lesssim R^2\{1+\delta_1^{-1}\|\nabla W\|_{L^2}^2\}\delta(t). 
\end{align*}
In the latter case, we use the fact that $W$ is real to deduce
\begin{align*}
|I_{R}[u(t)]|&\leq \left|2\IM\int_{\R^{3}}\nabla w_{R}(\bar{u} \nabla u-e^{-i\theta(t)}W_{[\lambda(t)]} 
\nabla [ e^{i\theta(t)}W_{[\lambda(t)]}])
dx\right|\\
&\lesssim
R^{2}[\|u(t)\|_{\dot{H}^{1}_{x}}+\|W\|_{\dot{H}^{1}}]
\|u(t)-e^{i\theta (t)}W_{[\lambda(t)]}\|_{\dot{H}^{1}}\\
&\lesssim_{W}R^{2} \|g(t)\|_{\dot{H}^{1}}   \lesssim_{W}R^{2} \delta(t),
\end{align*}
where in the last inequality we have used \eqref{Estimatemodu} with $\lambda(t):=\mu(t)$ for $t\in I_{0}$ and  the notation $W_{[\lambda(t)]}:=\lambda(t)^{\frac{1}{2}}W(\lambda(t)\cdot)$. \end{proof}

In the next lemma, we make use of the notation introduced in Section~\ref{S:virial}.

\begin{lemma}\label{LemmaB22}
Under assumptions of Lemma~\ref{Lemma11}, if $u$ is forward-global, then there exists $R_{1}\geq 1$,  such that for $R\geq R_{1}$ and $t\geq 0$ we have
\begin{align}\label{BNI}
	\tfrac{d}{dt}I_{R}[u(t)]&\leq -14 \delta(t).
\end{align}
\end{lemma}
\begin{proof}
Fix $R>1$ to be determined below. For $\delta_{1}\in (0, \delta_{0})$ sufficiently small, we use the localized virial identity (cf. Lemma~\ref{VirialModulate}) with $\chi(t)$ satisfying
\[
\chi(t)=
\begin{cases}
1& \quad \delta(t)<\delta_{1} \\
0& \quad \delta(t)\geq \delta_{1}.
\end{cases}
\]
We recall that $E^{c}(W)=\frac{1}{3}\|\nabla W\|^{2}_{L^{2}}$. By condition \eqref{BCon11} we  get
\begin{equation*}
{{F^{c}_{\infty}}}[u(t)]=-16\delta(t)-12\|u(t)\|^{4}_{L^{4}}.
\end{equation*}
Thus, by Lemma~\ref{VirialModulate} we  deduce
\begin{equation}\label{VirilaX}
\tfrac{d}{dt}I_{R}[u(t)]=F^{c}_{\infty}[u(t)]+\EE(t)=-16\delta(t)-12\|u(t)\|^{4}_{L^{4}}+\EE(t)
\end{equation}
with
\begin{equation}\label{Error11}
\EE(t)=
\begin{cases}
F_{R}[u(t)]-F^{c}_{\infty}[u(t)]& \quad  \text{if $\delta(t)\geq \delta_{1}$}, \\
F_{R}[u(t)]-F^{c}_{\infty}[u(t)]-\K[t]& \quad \text{if $\delta(t)< \delta_{1}$},
\end{cases}
\end{equation}
where
\begin{equation}\label{Error22}
\K(t):=F^{c}_{R}[e^{i\theta(t)}\lambda(t)^{\frac{1}{2}}W(\lambda(t)x)]-F^{c}_{\infty}[e^{i\theta(t)}\lambda(t)^{\frac{1}{2}}W(\lambda(t)x)]
\end{equation}
and $\lambda(t):=\mu(t)$ for $t\in I_{0}$. 
Now by \eqref{DecomU} we have ${u}_{[\theta(t), \lambda(t)]}=W+V$, where 
\[
{u}_{[\theta(t), \lambda(t)]}=e^{-i\theta(t)}\lambda(t)^{-\frac{1}{2}}u(t, \lambda(t)^{-1}\cdot)\qtq{and}\|V\|_{\dot{H}^{1}}\sim \delta(t).
\]

We now claim that
\begin{align}\label{InfP}
\lambda_{\text{inf}}:=\inf\left\{{\lambda(t):\ t\geq 0,\ \delta(t)\leq \delta_{1}}\right\}>0
\end{align}
for $\delta_{1}\in (0, \delta_{0})$ sufficiently small. Indeed, by the mass conservation we see that
\begin{align*}
	M({u}_{0})&\gtrsim \int_{|x|\leq \tfrac{1}{\lambda(t)}}|u(x,t)|^{2}dx
	=\tfrac{1}{\lambda(t)^{2}}\int_{|x|\leq 1}|u_{[\theta(t), \lambda(t)]}|^{2}dx\\
	&\gtrsim \tfrac{1}{\lambda(t)^{2}}\(\int_{|x|\leq 1}W^{2}dx-\int_{|x|\leq 1}|V|^{2}dx\).
\end{align*}
Since
\[
\|V(t)\|_{L^{2}(|x|\leq 1)}\lesssim \|V(t)\|_{L^{6}(|x|\leq 1)}  
\lesssim \|V(t)\|_{\dot{H}^{1}}\lesssim\delta(t),
\]
it follows that
\[
M({u}_{0})\gtrsim  \tfrac{1}{\lambda(t)^{2}}\(\int_{|x|\leq 1}W^{2}dx-C\delta^{2}(t)\).
\]
Choosing $\delta_{1}$ sufficiently small yields \eqref{InfP}.

Next we show that there exists $R_{\ast}$ so that for $R\geq R_{\ast}$,
\begin{align}
\label{EstimateE22}
|\EE(t)|\leq 2\delta(t)\quad &\text{ for $t\geq 0$ with $\delta(t)< \delta_{1}$}.
\end{align}
To simplify the notation, we set  $W(t)=e^{i\theta(t)}\lambda(t)^{\frac{1}{2}}W(\lambda(t)x)$. Then
\begin{align}\label{Decomp11}
\EE(t)&=-8\int_{|x|\geq R}[(|\nabla u|^{2}-(|\nabla W(t)|^{2})]\,dx
+8\int_{|x|\geq R}[|u|^{6}-|W(t)|^{6}]\,dx\\\label{Decomp33}
&\quad+\int_{R\leq |x| \leq 2R}(-\Delta \Delta w_{R})[|u|^{2}-|W(t)|^{2}]\,dx \\\label{Decomp44}
&\quad-\tfrac{4}{3}\int_{|x|\geq R}\Delta[w_{R}(x)](|u|^{6}-|W(t)|^{6})\,dx\\\label{Decomp55}
&\quad+4\RE\int_{|x|\geq R}[\overline{u_{j}} u_{k}-\overline{\partial_{j}W(t)} \partial_{k}W(t)]\partial_{jk}[w_{R}(x)]\,dx\\ \label{Decomp66}
&\quad+\int_{\R^{3}}\Delta[w_{R}(x)]|u|^{4}\,dx
\end{align}
for all $t\geq 0$ such that $\delta(t)<\delta_{1}$.

As $|\Delta \Delta w_{R}|\lesssim 1/|x|^{2}$, $|\partial_{jk}[w_{R}]|\lesssim 1$ and $|\Delta[w_{R}]|\lesssim 1$, we deduce that \eqref{Decomp11}--\eqref{Decomp55} can be estimated by terms of the form 
\[
\bigl\{\|u(t)\|_{\dot{H}_{x}^{1}(|x|\geq R)}+\|W(t)\|_{\dot{H}_{x}^{1}(|x|\geq R)} +
\|u(t)\|^{5}_{L_{x}^{6}(|x|\geq R)}+\|W(t)\|^{5}_{L_{x}^{6}(|x|\geq R)}
\bigr\}\|g(t)\|_{\dot{H}_{x}^{1}},
\]
where $ g(t)=e^{-i\theta(t)}[{u}(t)-W(t)]$. Moreover, since 
\[
\|W(t)\|_{\dot{H}_{x}^{1}(|x|\geq R)} \sim \tfrac{1}{R^{1/2}}\qtq{and}\|W(t)\|_{L_{x}^{6}(|x|\geq R)} \sim \tfrac{1}{R^{1/2}},
\] 
we deduce that
\begin{align*}
&\|u(t)\|_{\dot{H}_{x}^{1}(|x|\geq R)}+\|W(t)\|_{\dot{H}_{x}^{1}(|x|\geq R)} +
\|u(t)\|^{5}_{L_{x}^{6}(|x|\geq R)}+\|W(t)\|^{5}_{L_{x}^{6}(|x|\geq R)}\\
&\leq \delta(t)^{1/2}+\delta(t)^{5}+\tfrac{1}{(\lambda(t)R)^{1/2}}+\tfrac{1}{(\lambda(t)R)^{5/2}}.
\end{align*}
Thus, for $\delta_{1}\in (0, \delta_{0})$ sufficiently small and $R$ sufficiently large,
\[
|\eqref{Decomp11}|+|\eqref{Decomp33}|+|\eqref{Decomp55}|\leq \delta(t) \qtq{when $\delta(t)<\delta_{1}$.}
\]
By taking $\delta_{1}$ smaller if necessary, Proposition~\ref{Modilation11} implies
\[
\left|\int_{\R^{3}}\Delta[w_{R}(x)]|u|^{4}dx\right|\lesssim \delta(t)^{2}\lesssim \delta_{1} \delta(t)\leq \delta(t),
\]
Putting together the estimates above yields \eqref{EstimateE22}.

Now, suppose that $\delta(t)\geq \delta_{1}$. We show that for $R$ large, we have
\begin{align}\label{EstimateE11}
\EE(t)\leq \delta(t) +6\|u(t)\|^{4}_{L^{4}}.
\end{align}
First, we recall the radial Sobolev embedding estimate:
\[
\|f\|_{L^{\infty}(|x|\geq R)}\lesssim \tfrac{1}{R}\|f\|^{1/2}_{L^{2}}\|\nabla f\|^{1/2}_{L^{2}}.
\]
For $\delta(t)\geq \delta_1$, we may write
\begin{align}\label{FF11}
\EE(t)=&\int_{|x|\geq R}(- \Delta \Delta w_{R})|u|^{2}
+4\RE \overline{u_{j}} u_{k} \partial_{jk}[w_{R}]-8|\nabla u|^{2}dx\\ \label{FF22}
&-\tfrac{4}{3}\int_{|x|\geq R}\Delta[w_{R}(x)]|u|^{6}dx
+8\int_{|x|\geq R}|u|^{6}dx
\\ \label{FF33}
&+\int_{\R^{3}}\Delta[w_{R}(x)]|u|^{4}dx.
\end{align}
Choosing $w_{R}$  so that $4\partial^{2}_{r}w_{R}\leq 8$, we have
\[
\int_{|x|\geq R}(4\RE \overline{u_{j}} u_{k} \partial_{jk}[w_{R}]-8|\nabla u|^{2})dx
\leq 0.
\]
Moreover, since $|\Delta \Delta w_{R}|\lesssim 1/|x|^{2}$, by \eqref{FF11}--\eqref{FF33} and radial Sobolev, 
\[
\EE(t)\leq \tfrac{C}{R^{2}}+\tfrac{C}{R^{4}}\|u(t)\|^{4}_{\dot{H}^{1}}
+\int_{\R^{3}}\Delta[w_{R}(x)]|u|^{4}dx.
\]
Furthermore, 
\begin{align*}
\int_{\R^{3}}\Delta[w_{R}(x)]|u|^{4}dx&=6\int_{|x|\leq R}|u|^{4}dx+\int_{|x|\geq R}\Delta[w_{R}(x)]|u|^{4}dx\\
&\leq 6\int_{\R^{3}}|u|^{4}dx+\tfrac{C}{R^{2}}\|u(t)\|_{\dot{H}^{1}}.
\end{align*}
 Combining estimates above we deduce (in the case $\delta(t)\geq \delta_{1}$)
\begin{align*}
	\EE(t)&\leq  6\int_{\R^{3}}|u|^{4}dx+ \tfrac{C}{R^{2}}+\tfrac{C}{R^{4}}\|u(t)\|^{4}_{\dot{H}^{1}}
+\tfrac{C}{R^{2}}\|u(t)\|_{\dot{H}^{1}}\\
&\leq 6\int_{\R^{3}}|u|^{4}dx+\tfrac{C(\delta_{1}, \|W\|_{\dot{H}^{1}})}{R^{2}}\delta(t).
\end{align*}
Thus \eqref{EstimateE11} holds for $R$ large.  Combining \eqref{EstimateE22} and \eqref{EstimateE11} now yields \eqref{BNI}.\end{proof}

\begin{lemma}\label{NeD}
Given ${u}$ as in Lemma~\ref{LemmaB22} , there exists $C>1$ and $c>0$ such that
\begin{align}\label{DoQ11}
\int^{\infty}_{t}\delta(s)\,ds \leq Ce^{-ct}\qtq{for all $t\geq 0$.}
\end{align}
\end{lemma}
\begin{proof}
Fix $R\geq R_{1}$ (where $R_1$ is as in Lemma~\ref{LemmaB22}). Writing
\begin{align}\label{VrV}
	V_{R}(t)=\int_{\R^{3}}w_{R}(x)|u(t,x)|^{2}\,dx,
\end{align}
we have that $\tfrac{d}{dt}V_{R}(t)=I_{R}[{u}]$. By Lemma~\ref{LemmaB22}, we have 
\[
\tfrac{d^{2}}{dt^{2}}V_{R}(t)=\tfrac{d}{dt}I_{R}[u(t)]\leq -14\delta(t).
\] 
Thus, since $\tfrac{d^{2}}{dt^{2}}V_{R}(t)<0$ and $V_{R}(t)>0$ for all $t\geq0$, it follows that  $I_{R}[u(t)]=\tfrac{d}{dt}V_{R}(t)>0$ for all $t\geq 0$. Thus \eqref{TBu} implies
\[
14\int^{T}_{t}\delta(s)\,ds\leq -\int^{T}_{t}\tfrac{d}{ds}I_{R}[u(s)]\,ds
=I_{R}[u(t)]-I_{R}[u(T)]\leq I_{R}[u(t)]\leq CR^{2}\delta(t).
\]
Sending $T\to \infty$, \eqref{DoQ11} now follows from Gronwall's inequality.
 \end{proof}

As an immediate consequence of Lemma~\ref{NeD}, we have the following result.

\begin{corollary}\label{delace} Under the assumptions of Lemma~\ref{Lemma11}, there exists an increasing sequence $t_{n}\to \infty$ such that $\delta(t_{n})\to 0$.
\end{corollary}

\begin{proof}[{Proof of Theorem~\ref{Th2}}~(ii)]
Let $u$ be a solution to \eqref{NLS} as in the statement of Theorem~\ref{Th2}~(ii). Suppose that $u$ is forward global.  By Corollary~\ref{delace}, there exists increasing $t_{n}\to \infty$ such that $\delta(t_{n})\to 0.$  Moreover, by Proposition~\ref{Modilation11} we deduce that (with $\lambda(t)=\mu(t)$ for $t\in I_{0}$)
\begin{align}\label{CWQ}
	e^{-i\theta(t_{n})}\lambda(t_{n})^{-\frac{1}{2}}u(t_{n}, \lambda(t_{n})^{-1}x)
\to W(x) \qtq{in $\dot{H}^{1}$.}
\end{align}
Passing to a subsequence, we may assume that
\[
\lim_{n\to +\infty}\lambda(t_{n})=\lambda_{\ast} \in [0, \infty].
\]
We show that $\lambda_{\ast}<\infty$.  Suppose instead that $\lambda_{\ast}=\infty$. From \eqref{CWQ} and a change of variables we get that for any $C>0$,
\[
\|u(t_{n})\|_{{L}^{6}(|x|\geq C)}\to 0\qtq{as $n\to +\infty$.}
\]
Fix $R\geq R_{1}$ (cf. Lemma~\ref{LemmaB22}) and $\epsilon>0$ and recall \eqref{VrV}. By Hardy and H\"older, 
\begin{align*}
V_{R}(t_{n})&=\int_{|x|\leq \epsilon}w_{R}(x)|u(t_n,x)|^{2}dx+\int_{\epsilon \leq |x|\leq 2R}w_{R}(x)|u(t_n,x)|^{2}dx\\
&\lesssim \epsilon^4 \|u(t_n)\|_{\dot H^1}^2 + R^4\|u(t_n)\|_{L^6(|x|\geq \epsilon)}^2.
\end{align*}
Sending $n\to \infty$ and then  $\epsilon \to 0$, we find that
\[
\lim_{n\to \infty}V_{R}(t_{n})=0.
\]
However, as $\tfrac{d}{dt}V_{R}>0$ for all $t\geq 0$ (see the proof of Lemma~\ref{NeD}), it follows that $V_{R}(t)<0$ for $t\geq 0$, a contradiction.  Thus $\lambda_{\ast}<\infty$. 

As $\delta(t_{n})\to 0$ as $n\to \infty$, there exists $N\in\N$ so that $\delta(t_{n})<\delta_{0}$ for $n\geq N$. Now observe that Proposition~\ref{Modilation11} yields
\[
\frac{1}{1+\lambda(t_{n})}\leq C\delta(t_{n}) \qtq{for all $n\geq N$.}
\]
Taking $n\to \infty$, we obtain a contradiction, so that $u$ blows up forward in time.  Using time-reversal symmetry, we can also obtain that $u$ blows up backward in time.\end{proof}


\section{Compactness for nonscattering solutions}\label{S:Compactness}

We next show that the failure of Theorem~\ref{Th2}~(i) implies the existence of a nonscattering solution $u$ to \eqref{NLS} satisfying \eqref{Thres} such that the orbit of $u$ is precompact in $\dot{H}^{1}(\R^{3})$ modulo scaling.

\begin{proposition}\label{Compacness11}
Suppose Theorem~\ref{Th2}~(i) fails. Then there exists a solution ${u}: [0, T^{\ast})\times \R^{3}\to \C$ of \eqref{NLS} with
\begin{align}\label{EGN}
	&E(u_{0})=E^{c}(W) \quad \text{and} \quad \|\nabla u_{0}\|_{L^{2}}<\|\nabla W\|_{L^{2}}\\
	\label{SN}
	&\|{u}\|_{L^{10}_{t, x}([0, T^{\ast})\times \R^{3})}=\infty.
\end{align}
Moreover, there exists $\lambda:[0, T^{\ast}) \to (0, \infty)$ such that
 \begin{equation}\label{CompactX}
\left\{\lambda(t)^{-\frac{1}{2}}u(t, \tfrac{x}{\lambda(t)}): t\in [0, T^{\ast})\right\} \quad
\text{is pre-compact in $\dot{H}^{1}(\R^{3})$}.
\end{equation}
\end{proposition}
\begin{proof}
As the proof follows that of \cite[Proposition 9.1]{KillipOhPoVi2017}, we will only sketch the main steps:  If Theorem~\ref{Th2}~(i) fails, then  there exists a radial solution ${u}: [0, T^{\ast})\times \R^{3}\to \C$ 
obeying \eqref{EGN} and \eqref{SN}. In particular, by Lemma~\ref{GlobalS},
\begin{align}\label{boundGC}
 &\|\nabla u(t)\|_{L^{2}}<\|\nabla W\|_{L^{2}}\qtq{for all $t\in [0, T^{\ast})$.}
\end{align}
\textbf{Step 1.} For any sequence $\left\{\tau_{n}\right\}_{n\in \N}\subset [0, T^{\ast})$, there exists $\lambda_{n}$ such that 
\begin{align}\label{LamC}
\lambda_{n}^{-\frac{1}{2}}u\(\tau_{n}, \tfrac{x}{\lambda_{n}}\)
\qtq{converges strongly in $\dot{H}^{1}$ (up to a subsequence).}	
\end{align}
By continuity of $u$ it suffices to consider $\tau_{n}\to T^{\ast}$. Using \eqref{boundGC}, we apply the profile decomposition (Theorem~\ref{Profi}) to $\left\{u(\tau_{n})\right\}_{n\in \N}$ and write
\begin{equation}\label{Dpe}
u_{n}:=u(\tau_{n})=\sum^{J}_{j=1}\phi^{j}_{n}+R_n^J
\end{equation}
with $J\leq J^{\ast}$ and all of the properties stated in Theorem~\ref{Profi} and Lemma~\ref{Lde}. 

We will show that $J^{\ast}=1$ by showing that all other possibilities contradict \eqref{SN}.

If $J^{\ast}=0$, then by \eqref{Sr} we have that
\[
\| e^{it\Delta}u(\tau_{n})  \|_{L^{10}_{t,x}(\left\{t>0\right\}\times \R^{3})}
\to 0 \qtq{as $n\to \infty$.}
\]
Applying Lemma~\ref{stabi} (with $\widetilde{u}=e^{it\Delta}u(\tau_{n}) $), we obtain that for $n$ large,
\[\|
u(t+\tau_{n})\|_{L^{10}_{t,x}(\left\{t>0\right\}\times \R^{3})}
=\|u  \|_{L^{10}_{t,x}(\left\{t>\tau_{n}\right\}\times \R^{3})}\lesssim 1,
\] 
contradicting \eqref{SN}.

Next suppose $J^{\ast}\geq 2$. By Theorem~\ref{Profi}, Lemma~\ref{Lde} and \eqref{boundGC} we have 
\begin{align}\label{MassC}
&\lim_{n \to \infty} \big(\sum_{j=1}^{J} M(\phi_n^j) + M(R_n^J)\big) = \lim_{n \to \infty} M(u_{n}) =M(u_0),\\
\label{DECE}
&\lim_{n \to \infty} \big(\sum_{j=1}^{J} E(\phi_n^j) + E(R_n^J)\big) = \lim_{n \to \infty} E(u_{n})=E^{c}(W),\\
\label{DEG}
&\limsup_{n \to \infty} \big(\sum_{j=1}^{J} \| \phi_n^j\|_{\dot{H}^{1}}^2 + \| R_n^J\|_{\dot{H}^{1}}^2\big)
	= \limsup_{n \to \infty} \|u(\tau_n)\|_{\dot{H}^{1}}^2 \leq \|W\|_{\dot{H}^{1}}^2
\end{align}
for any $0\leq J \leq J^{\ast}$. By \eqref{IneW} and \eqref{DEG} we deduce that for $n$ large,
\begin{align}\label{BounWd}
 \| \phi_n^j\|_{\dot{H}^{1}}^2\leq \frac{\|W\|_{\dot{H}^{1}}^2}{E^{c}(W)} E(\phi_n^j)
\qtq{and}
 \| R_n^J\|_{\dot{H}^{1}}^2\leq \frac{\|W\|_{\dot{H}^{1}}^2}{E^{c}(W)} E(R_n^J).
\end{align}
In particular, $\liminf_{n\to \infty}E(\phi_n^j)>0$. Thus there exists  $\delta>0$ so that
\begin{align}\label{MEC}
	E(\phi_n^j)&\leq E^{c}(W)-\delta,
\end{align}
for sufficiently large $n$ and $1\leq j\leq J$, so that each $ \phi_n^j$ satisfies \eqref{Thres00} in Theorem~\ref{Th1}. 

We use the $\phi_n^j$ to build approximate solutions to \eqref{NLS} in the following three cases: 
\begin{itemize}
\item $\lambda^{j}_{n}\equiv 1$ and $t^{j}_{n}\equiv 0$; 
\item $\lambda^{j}_{n}\equiv 1$ and $t^{j}_{n}\to \pm \infty$; 
\item and $\lambda^{j}_{n}\to +\infty$.
\end{itemize} 
For $j$ such that $\lambda^{j}_{n}\equiv 1$ and $t^{j}_{n}\equiv 0$, \eqref{MEC} implies that we may apply Theorem~\ref{Th1}~(i) to construct a global solution $v^{j}$ with initial data $v^{j}(0)=\phi^{j}$ obeying global space-time bounds. For $j$ such that $\lambda^{j}_{n}\equiv 1$ and $t^{j}_{n}\to \pm \infty$, we take the solution $v^{j}$ to \eqref{NLS} which scatters to $e^{it\Delta}\phi^{j}$ in $H^{1}$ as  $t\to \pm \infty$. By \eqref{MEC} and Theorem~\ref{Th1}~(i) we have that the solution $v^{j}$ is global and satisfies uniform space-time bounds. In either case, we set
\[
v^{j}_{n}(t,x)=v^{j}(t+t^{j}_{n}, x)
\]
with
\[
\| v^{j}_{n}  \|_{L^{10}_{t,x}(\R\times\R^{3})}\lesssim_{\delta}1.
\]
By persistence of regularity (Remark~\ref{Rpr}) and \eqref{BounWd} we get
\begin{equation}\label{Estrii}
\| v^{j}_{n}  \|_{L^{10}_{t,x}(\R\times\R^{3})}+\|\nabla  v^{j}_{n}\|_{L^{\frac{10}{3}}_{t,x}(\R\times\R^{3})}\lesssim_{\delta} E(v^{j}_{n})^{\frac{1}{2}},
\end{equation}
\begin{equation}\label{masses}
\|  v^{j}_{n}\|_{L^{\frac{10}{3}}_{t,x}(\R\times\R^{3})}\lesssim_{\delta} M(v^{j}_{n})^{\frac{1}{2}}.
\end{equation}
Finally, for the case $\lambda^{j}_{n}\rightarrow +\infty$ as $n\rightarrow\infty$, we define $v_{n}^{j}$ to be the global solution to equation \eqref{NLS} with data $v_{n}^{j}(0)=\phi_{n}^{j}$ given by Lemma~\ref{P:embedding}, which satisfies
\begin{align}\label{Lamdace}
\|  v^{j}_{n}\|_{L^{\frac{10}{3}}_{t,x}(\R\times\R^{3})}+	\| v^{j}_{n}  \|_{L^{10}_{t,x}(\R\times\R^{3})}
\lesssim_{\delta} E(v^{j}_{n})^{\frac{1}{2}}, \quad
\|\nabla  v^{j}_{n}\|_{L^{\frac{10}{3}}_{t,x}(\R\times\R^{3})}
\lesssim M(v^{j}_{n})^{\frac{1}{2}}.
\end{align}

In addition, by \eqref{DECE} and \eqref{MassC}  we obtain for each finite $J\leq J^{\ast}$,
\begin{align}\label{PTGq}
&\limsup_{n\to+\infty}\sum^{J}_{j=1}E(v^{j}_{n} )+E(R_n^J)\leq E^{c}(W),\\ \label{PTGq22}
&\limsup_{n\to+\infty}\sum^{J}_{j=1}M(v^{j}_{n} )+M(R_n^J)\leq M(u_{0}).
\end{align}

In either case,
\begin{equation}\label{Apro11}
\|v_{n}^{j}(0)- \phi_{n}^{j}\|_{H^{1}_{x}}\rightarrow0 \quad \text{as $n\rightarrow\infty$}.
\end{equation}

Next, we define approximate solutions of \eqref{NLS} by
\[u^{J}_{n}(t,x)=\sum^{J}_{j=1}v^{j}_{n}(t,x)+e^{it\Delta}R_n^J,\]
and we will use Lemma \ref{stabi} to contradict \eqref{SN}.

First notice that by construction (see \eqref{Dpe} and \eqref{Apro11}) we get
\begin{equation}\label{Aps}
\|u^{J}_{n}(0)-u_{n}\|_{H^{1}_{x}}\rightarrow0,\quad \text{as $n\rightarrow\infty$}.
\end{equation}

Moreover, with estimates \eqref{Small11}, \eqref{Small22}, \eqref{BounWd}, \eqref{Estrii}, \eqref{masses} and \eqref{Lamdace} in hand, the arguments of \cite[Lemmas 9.4 and 9.5]{KillipOhPoVi2017} imply:

\begin{enumerate}[label=\rm{(\roman*)}]
\item We have the following global space bound (cf. \cite[Lemma 9.4]{KillipOhPoVi2017}):
\begin{equation}\label{Gstb}
\sup_{J}\limsup_{n\rightarrow\infty}\left[ \| {u}^{J}_{n}  \|_{L^{10}_{t,x}(\R\times \R^{3})} 
+\| {u}^{J}_{n}  \|_{L_{t}^{\frac{10}{3}}H^{1,\frac{10}{3}}_{x}(\R\times \R^{3})} +
\| u^{J}_{n}  \|_{L_{t}^{\infty}H^{1}_{x}(\R\times\R^{3})}
\right]\lesssim 1.
\end{equation}

\item Let $\epsilon>0$. For large $J$ we have (cf. \cite[Lemma 9.5]{KillipOhPoVi2017})
\begin{equation}\label{Eimpo}
\limsup_{n\rightarrow\infty}\|\nabla  e^{J}_{n}  \|_{L^{\frac{10}{7}}_{x}(\R\times \R^{3})}\leq\epsilon,
\end{equation}
\[
e^{J}_{n}:=(i\partial_{t}+\Delta) {u}^{J}_{n}+|{u}^{J}_{n}|^{4}{u}^{J}_{n}-|{u}^{J}_{n}|^{2}{u}^{J}_{n}.
\]
\end{enumerate}

Combining \eqref{Aps}, \eqref{Gstb} and \eqref{Eimpo}, Lemma~\ref{stabi} implies that $u\in L^{10}_{t, x}(\R\times\R^3)$, contradicting \eqref{SN}. 

Having established $J^*=1$, \eqref{Dpe} simplifies to
\begin{align}\label{1Dec}
	u_{n}=u(\tau_{n})=\phi_{n}+W^{1}_{n}.
\end{align}
We now observe that 
\begin{align}\label{ZerI}
	\|\nabla W^{1}_{n}\|^{2}_{L^{2}}\to 0 \qtq{as $n\to +\infty$.}
\end{align}
 Indeed, otherwise, by \eqref{BounWd} we see that there exists $c>0$ so that  $E( W^{1}_{n})\geq c$. Thus, $\phi_{n}$ obeys the sub-threshold condition \eqref{MEC}, and the same arguments used above show $u\in L^{10}_{t, x}(\R\times \R^{3})$, contradicting \eqref{SN}. 

Now we preclude the possibility that $t_{n}\rightarrow \pm\infty$ as $n\rightarrow\infty$. Without loss of generality, suppose $t_{n}\rightarrow \infty$.  If $\lambda_{n}\equiv1$, then by monotone convergence we have
\[
\| e^{it\Delta}\phi_{n}  \|_{L^{10}_{t,x}([0,\infty)\times \R^{3})}= \|e^{it\Delta} \phi\|_{L^{10}_{t,x}([t_{n},\infty)\times \R^{3})}\rightarrow0
\]
as $n\rightarrow\infty$. On the other hand, if $\lambda_{n}\to +\infty$ a change of variables, \eqref{Ap11}, and  monotone convergence yield
\[
\| e^{it\Delta}\phi_{n}  \|_{L^{10}_{t,x}([0,\infty)\times \R^{3})}=\|e^{it\Delta} P_{\geq (\lambda_{n})^{-\theta}}\phi\|_{L^{10}_{t,x}([t_{n},\infty)\times \R^{3})}\rightarrow0,
\]
as $n\rightarrow\infty$. In either case, we may apply Lemma~\ref{stabi} with $\tilde{u}:=e^{it\Delta}\phi_{n}$ to obtain that $u\in L^{10}_{t,x}([0,\infty)\times \R^{3})$, contradicting \eqref{SN}.  Thus $t_n\equiv 0$.

Finally, by \eqref{1Dec} and \eqref{ZerI} we get
\[
\|\lambda^{-\frac{1}{2}}_{n}u(\tau_{n}, \tfrac{x}{\lambda_{n}})-\phi\|^{2}_{\dot{H}^{1}}\to 0 \qtq{as $n\to \infty$.}
\]
Note that in the case $\lambda_{n}\to +\infty$, we have used the fact that 
$\|P_{\geq (\lambda_{n})^{-\theta}}\phi-\phi\|_{\dot{H}^{1}}\to 0$ as $n\to \infty$. This complete the proof of \eqref{LamC}.

\textbf{Step 2.} We define $\lambda(t)$. By Gagliardo--Nirenberg we deduce that there exists $\widetilde{C}>0$ so that
\[
2E^{c}(W)=2E(u(t))\leq \widetilde{C}\| u(t)\|^{2}_{\dot{H}^{1}}\leq \widetilde{C}\| W\|^{2}_{\dot{H}^{1}}
\]
for $t$ in the lifespan of $u$. Thus we may define a function $\lambda:[0, T^{\ast})\to \R^{+}$ such that
\begin{align}\label{DefiLa}
	\|u(t)\|^{2}_{\dot{H}^{1}(|x|\leq \frac{1}{\lambda(t)})}=
	\sup_{\lambda>0}	\|u(t)\|^{2}_{\dot{H}^{1}(|x|\leq \frac{1}{\lambda})}=\tfrac{1}{\widetilde{C}}E^{c}(W).
\end{align}
Now we verify \eqref{CompactX}. Choose an arbitrary sequence $t_{n}$. By \eqref{LamC} there exist a nonzero $\psi\in \dot{H}^{1}$ and $\lambda_{n}\geq1$
such that
\begin{align}\label{LamC11}
\lambda_{n}^{-\frac{1}{2}}u\(t_{n}, \tfrac{x}{\lambda_{n}}\)\to \psi
\qtq{in $\dot{H}^{1}$.}	
\end{align}
In particular, up to subsequence, we have
\begin{align}\label{FIC}
	\lambda(t_{n})^{-\frac{1}{2}}u\(t_{n}, \tfrac{x}{\lambda(t_{n})}\)=
	\(\tfrac{\lambda_{n}}{\lambda(t_{n})}\)^{\frac{1}{2}}\psi\( \tfrac{\lambda_{n}}{\lambda(t_{n})}x\)+o(1)
	\qtq{in $\dot{H}^{1}$.}
\end{align}
Combining \eqref{DefiLa}, \eqref{LamC11} and \eqref{FIC} we obtain that there exists $C>0$ so that 
\begin{align}\label{LiLa}
\tfrac{\lambda_{n}}{\lambda(t_{n})}+\tfrac{\lambda(t_{n})}{\lambda_{n}}\leq C \qtq{for all $n\in \N$.}
\end{align}
Thus, by \eqref{FIC} and \eqref{LiLa} we obtain  \eqref{CompactX}.\end{proof}

\section{Precluding the compact solution}\label{S:Impossibility}

In this section we complete the proof of Theorem~\ref{Th2}(i). 

\subsection{Finite-time blowup}\label{S:FTB}
\begin{theorem}\label{Tfinito}
There are no solutions of \eqref{NLS} as in Proposition~\ref{Compacness11} with $T^{\ast}<\infty$.
\end{theorem}
\begin{proof}

Suppose that $u:[0, T^{\ast})\times \R^{3}\to \C$ were such a solution. We claim
\begin{align}\label{LamdaI}
\lim_{t\to T^{\ast}}\lambda(t)=\infty.
\end{align}
Suppose instead that \eqref{LamdaI} fails. Then there exists a sequence $t_{n}$ that converges to $T^{\ast}$ such that $\lambda(t_{n})\to \lambda_{0}\in \R$. By \eqref{LiLa} we infer that $\lambda_{0}>0$. Thus, by \eqref{CompactX}, we obtain that there exists nonzero $\psi\in \dot{H}^{1}$ such that
\begin{align}\label{CVs}
	\|u(t_{n})-\psi\|_{\dot{H}^{1}}\to 0 \qtq{as $n\to\infty$.}
\end{align}
Moreover, by the mass conservation and the uniqueness of the weak limit we have $\|\psi\|^{2}_{L^{2}}\leq M(u_{0})$.  By \eqref{CVs}, the stability result (Lemma~\ref{stabi}), and following the argument in \cite[Proposition 5.3, Case 1]{KenigMerle2006} we get that 
\[
\| u  \|_{S^{0}([T^{\ast}-\epsilon, T^{\ast}+\epsilon]\times \R^{3})}<\infty
\]
for some $\epsilon>0$ sufficiently small, which contradicts \eqref{SN} and hence proves \eqref{LamdaI}.

Next, fix $R>0$. For $t\in [0, T^\ast)$ we define 
\[
M_{R}(t)=\int_{\R^{3}}|u(t,x)|^{2}\xi\(\tfrac{x}{R} \)\,dx.
\qtq{}
\]
where $\xi(x)=1$ if $|x|\leq 1$ and $\xi(x)=0$ if $|x|\geq 2$. Since 
\[
M^{\prime}_{R}(t)=\tfrac{2}{R}\IM \int_{\R^{3}}\overline{u}\nabla u\cdot (\nabla\xi)\(  \tfrac{x}{R} \)dx,
\]
it follows from \eqref{PositiveP} that $M^{\prime}_{R}(t)\lesssim_{W} 1$. Thus
\begin{align}\label{InteF}
M_{R}(t_{1})=M_{R}(t_{2})+\int^{t_{2}}_{t_{1}}M^{\prime}_{R}(t)dt \lesssim_{W}M_{R}(t_{2}) +|t_{1}-t_{2}|.
\end{align}
On the other hand, for $\mu\in (0,1)$ and $t\in [0, T^{\ast})$,
\begin{align*}
\int_{|x|\leq R}|u(t,x)|^{2}dx
&\leq  
\int_{|x|\leq \mu R}|u(t,x)|^{2}dx+\int_{\mu R\leq |x|\leq R}|u(t,x)|^{2}dx\\
&\lesssim
\mu^{2} R^{2}\|u(t)\|^{2}_{L^{6}}+R^{2}\(\int_{|x|\geq \eta R}|u(t,x)|^{6}dx\)^{1/3}\\
&\lesssim \mu^{2} R^{2}\|\nabla W\|^{2}_{L^{2}}+R^{2}\(\int_{|x|\geq \eta R}|u(t,x)|^{6}dx\)^{1/3}.
\end{align*}
By \eqref{CompactX} and \eqref{LamdaI} we see that for all $R_{0}>0$,
\[\int_{|x|>R_{0}}|u(t,x)|^{6}\,dx\to 0 \qtq{as $t\to T^{\ast}$.}\]
Thus we deduce that for any $R>0$,
\[
\limsup_{t\nearrow T^{\ast}}\int_{|x|\leq R}|u(t,x)|^{2}dx=0.
\]
In particular, $M_{R}(t_{2})\to 0$ as $t_{2}\to T^{\ast}$, so that \eqref{InteF} implies
\[
M_{R}(t_{1}) \lesssim_{W}|T^{{\ast}}-t_{1}|.
\]
Letting $R\to\infty$ and using conservation of mass, we get $M(u_{0})\lesssim_{W}|T^{{\ast}}-t_{1}|$. Letting $t_{1}\to T^{\ast}$ we obtain ${u}_{0}=0$, which contradicts $E({u}_{0})= E^{c}(W)>0$.\end{proof}

\begin{remark}\label{FirRe} Combining the results established so far with existing results for NLS with combined nonlinearities (cf. \cite{TaoVisanZhang2007}), we can deduce that if $u$ is a maximal-lifespan solution to \eqref{NLS} satisfying 
\[
E({u}_{0})\leq E^{c}(W) \qtq{and} \|\nabla u_{0}\|_{L^{2}}\leq \|\nabla W\|_{L^{2}},
\]
then $u$ is global in time.
\end{remark}
%
%
%
%
%
%

\subsection{Infinite time blowup}\label{S:ITB}We now exclude the possibility of infinite time blowup.

\begin{theorem}\label{Tfinito22}
There are no solutions of \eqref{NLS} as in Proposition~\ref{Compacness11} with $T^{\ast}=\infty$.
\end{theorem}

We prove Theorem~\ref{Tfinito22} by contradiction.  Thus for the rest of this section, we suppose that $u:[0, \infty)\times \R^{3}\to \C$ is a solution to \eqref{NLS} as in Proposition~\ref{Compacness11} with $T^\ast=\infty$. In particular,  $u$ satisfies
\begin{align}\label{Eequal}
E(u_{0})=E^{c}(W), \quad \|\nabla u_{0}\|_{L^{2}}<\|\nabla W\|_{L^{2}},\quad \text{and}\quad  \|u \|_{L_{t,x}^{10}([0, \infty)\times\R^{3})}=\infty.
\end{align}
Moreover,  there exists a function  $\lambda_{0}:[0, \infty) \to (0,\infty)$ such that the set
 \begin{equation}\label{New0Compact}
\left\{\lambda_{0}(t)^{-\frac{1}{2}}u(t, \tfrac{x}{\lambda_{0}(t)}): t\in [0, +\infty)\right\} \quad
\text{is pre-compact in $\dot{H}^{1}(\R^{3})$}.
\end{equation}

\begin{lemma}\label{Parametrization}
If $\delta_{0}$ is sufficiently small, then there exists $C>0$ so that
\begin{align}\label{limia}
	\tfrac{\lambda_{0}(t)}{\mu(t)}+\tfrac{\mu(t)}{\lambda_{0}(t)}\leq C \quad \text{for  $t\in I_{0}$},
\end{align}
where the parameter $\mu(t)$ is given in Proposition~\ref{Modilation11}.
\end{lemma}

\begin{proof}
By \eqref{DecomU} and \eqref{Estimatemodu} we deduce for $\delta_{0}\ll 1$ and $t\in I_{0}$,
\[
\int_{\frac{1}{\mu(t)}\leq|x|\leq\frac{2}{\mu(t)}}|\nabla u(t,x)|^{2}dx
\geq \tfrac{1}{2}\int_{1\leq|x|\leq 2}|\nabla W(x)|^2dx-C\delta^{2}(t)\geq c
\]
for some positive constant $c$. By a change of variable we have
\begin{align}\label{PoW}
\int_{\frac{\lambda_{0}(t)}{\mu(t)}\leq|x|\leq\frac{2\lambda_{0}(t)}{\mu(t)}}\tfrac{1}{\lambda_{0}(t)^{3}}
\left|\nabla u\(t,\tfrac{x}{\lambda_{0}(t)}\)\right|^{2}dx\geq c\qtq{for all}t\in I_0.
\end{align}
Thus \eqref{limia} follows from compactness (cf. \eqref{New0Compact}).
\end{proof}

Using Lemma~\ref{Parametrization}, we deduce that
\begin{equation}\label{CompacNew}
\left\{\lambda(t)^{-\frac{1}{2}}u(t, \tfrac{x}{\lambda(t)}): t\in [0, \infty)\right\} \quad
\text{is pre-compact in $\dot{H}^{1}(\R^{3})$},
\end{equation}
where
\[
\lambda(t)=
\begin{cases}
\lambda_{0}(t)& \quad t\in [0, \infty)\setminus I_{0},\\
\mu(t)&\quad t\in I_{0}.
\end{cases}
\]

We now prove several lemmas relating the behavior of the scale function, $L^4$ portion of the energy, the function $\delta(t)$, and the virial functional.

\begin{lemma}\label{SequenceInf}
For any sequence $\left\{t_{n}\right\}\subset [0, \infty)$,  the following holds
\begin{equation}\label{ZeroPoten}
\lambda(t_{n})\to \infty \quad \text{if and only if}\quad \int_{\R^{3}}|u(t_{n},x)|^{4}dx\to 0.
\end{equation}
\end{lemma}
\begin{proof}
Suppose  $\lambda(t_{n})\to \infty$ and $\epsilon>0$. Using \eqref{CompacNew}, there exists $\rho_{\epsilon}>0$ so that
\[
\int_{{|x|\geq \tfrac{\rho_{\epsilon}}{\lambda(t_{n})}}}\bigl[|\nabla u(t_{n},x)|^{2}+|u(t_{n},x)|^{6}\bigr]dx\ll\epsilon \qtq{for all $n\in \N$.}
\]
As $M(u(t_{n}))=M(u_{0})$, by interpolation we obtain
\begin{align}\label{PNe}
	\int_{{|x|\geq \tfrac{\rho_{\epsilon}}{\lambda(t_{n})}}}| u(t_{n},x)|^{4}dx <\epsilon \qtq{for all $n\in \N$.}
\end{align}
Next, since $\tfrac{1}{\lambda(t_{n})}\to 0$ as $n\to\infty$, we deduce from dominated convergence that
\begin{align}\label{Elar}
\int_{{|x|\leq \tfrac{\rho_{\epsilon}}{\lambda(t_{n})}}}| u(t_{n},x)|^{4}dx  \to 0 \qtq{as $n\to \infty$.}
\end{align}
Combining \eqref{PNe} and \eqref{Elar} we have that $\int_{\R^{3}}|u(t_{n},x)|^{4}dx$ tends to zero as $n\to +\infty$.

On the other hand, suppose that
\begin{align}\label{Otra}
\int_{\R^{3}}|u(t_{n},x)|^{4}dx\to 0,
\end{align}
but  (possibly for a subsequence only) $\left\{\lambda(t_{n})\right\}_{n\in\N}$ converges.  From \eqref{CompacNew} we see that there exists $\phi\in \dot{H}^{1}(\R^{3})$ so that
\begin{equation}\label{contra11}
u(t_{n})\to \phi\quad \text{in $\dot{H}^{1}(\R^{3})$}.
\end{equation}
Combining \eqref{Otra} and \eqref{contra11} we have $E^{c}(W)=E^{c}(\phi)$, so that $\phi\neq 0$. Moreover, as $M(u(t_{n}))=M(u_{0})$, we derive that $\phi\in L^{2}(\R^{3})$. 

However, by the Gagliardo--Nirenberg inequality and \eqref{contra11} we get
\[
\|u(t_{n})-\phi\|^{4}_{L^{4}}\lesssim \|u(t_{n})-\phi\|_{L^{2}}\|u(t_{n})-\phi\|^{3}_{\dot{H}^{1}}
\to 0 \qtq{as $n\to \infty$.}
\]
Thus, by \eqref{Otra} we have that $\|\phi\|^{4}_{L^{4}}=0$, a contradiction. \end{proof}

\begin{lemma}\label{DeltaZero}
If $t_{n}\to \infty$, then
\begin{equation}\label{equivalence}
\lambda(t_{n})\to \infty \quad \text{if and only if}\quad \delta(t_{n})\to 0.
\end{equation}
\end{lemma}
\begin{proof} If $\delta(t_{n})\to 0$,  then Proposition \ref{Modilation11} and 
Lemma \ref{SequenceInf} imply $\lambda(t_{n})\to \infty$.

Next assume towards a contradiction that $\lambda(t_{n})\to \infty$ but (possibly for a subsequence only) we have
\begin{equation}\label{Zeroplus}
\delta(u(t_{n}))\geq c>0.
\end{equation}
As 
\[
\left\{\lambda(t)^{-\frac{1}{2}}u(t, \tfrac{x}{\lambda(t)}): t\in [0, \infty)\right\} \quad
\text{is pre-compact in $\dot{H}^{1}(\R^{3})$},
\] 
there exists $v_{0}\in \dot{H}^{1}(\R^{3})$ such that (passing to a further subsequence if necessary)
\begin{equation}\label{CompactV}
\lambda(t_{n})^{-\frac{1}{2}}u(t_{n}, \tfrac{x}{\lambda(t_{n})}) \to v_{0} \quad \text{in}\quad \dot{H}^{1}(\R^{3}).
\end{equation}
In particular, from \eqref{Eequal}, \eqref{ZeroPoten} and \eqref{Zeroplus}  we have
\begin{align}\label{ScR}
	E^{c}(v_{0})=E^{c}(W) \quad \text{and}\quad \|\nabla v_{0}\|_{L^{2}}<\|\nabla W\|_{L^{2}}.
\end{align}

Now let $v$ be the solution of the energy-critical NLS  \eqref{ECgNLS} with$v|_{t=0}=v_{0}$. By \cite[Theorem 2]{DuyckaMerle2009}, $v$ is global, with 
\[
\|v\|_{L^{10}_{t,x}([0, \infty)\times \R^{3})}<\infty\qtq{or}\|v\|_{L^{10}_{t,x}((-\infty,0]\times \R^{3})}<\infty
\]
(or both). In either case, we will use stability theory to deduce that 
$u\in L^{10}_{t,x}([0, \infty)\times \R^{3})$, contradicting \eqref{Eequal}.

The argument is similar to the one used in the proof of Proposition~\ref{P:embedding}.  We suppose that $\|v\|_{L^{10}_{t,x}([0, +\infty)\times \R^{3})}<+\infty$. The case when $\|v\|_{L^{10}_{t,x}((-\infty,0]\times \R^{3})}<+\infty$ can be treated similarly. 

We let $w_{n}$ be the solution to the energy-critical NLS \eqref{ECgNLS} with data $w_{n}(0)=P_{\geq (\lambda_{n})^{-\theta}}v_{0}$ for some $0<\theta \ll1$, where
$\lambda_{n}:=\lambda(t_{n})$.  Since $\lambda_{n}\to \infty$ as $n\to\infty$, 
\begin{align}\label{cvv}
	\| P_{\geq (\lambda_{n})^{-\theta}}v_{0}-v_{0} \|_{\dot{H}^{1}_{x}}\rightarrow 0 \quad \text{as $n\rightarrow \infty$}.
\end{align}
By \eqref{cvv} and stability for \eqref{ECgNLS} (cf. \cite[Theorem 2.14]{KenigMerle2006}), we deduce that
\begin{align*}
	\|\nabla w_{n}   \|_{S^{0}([0, +\infty)\times \R^{3})}\lesssim 1\qtq{for $n$ large.}
\end{align*}
Using Lemma~\ref{BIN}, we obtain
\[
\| P_{\geq (\lambda_{n})^{-\theta}}v_{0}\|_{{L}^{2}}\lesssim\lambda^{\theta}_{n}\| \nabla v_{0}\|_{{L}^{2}},
\]
which implies (by persistence of regularity)
\begin{align}\label{wpersi}
	\|w_{n}\|_{S^{0}([0, +\infty)\times \R^{3})}\leq C(\|v_{0}\|_{\dot{H}^{1}_{x}})\lambda^{\theta}_{n}.
\end{align}
Arguing as in  Proposition~\ref{P:embedding} one can show that
\[
\tilde{u}_{n}(t,x)=\lambda^{\frac{1}{2}}_{n}w_{n}(\lambda^{2}_{n}t, \lambda_{n}x)
\]
are approximate solutions of equation \eqref{NLS} on $[0, \infty)$. Indeed, following the same argument as \eqref{Inter1} we see that
\[
\lim_{n\to +\infty}\|\nabla\bigl( |\tilde{u}_{n}|^{2}\tilde{u}_{n}\bigr) \|_{L^{\frac{10}{7}}_{t, x}(\R\times \R^{3})}=0.
\]
Moreover, combining \eqref{CompactV} and \eqref{cvv}  we get
\[
\| \tilde{u}_{n}(0)-u(t_{n}) \|_{\dot{H}^{1}_{x}}\rightarrow 0 \quad \text{as $n\rightarrow \infty$}.
\]
Thus we may apply Lemma \ref{stabi} to deduce that for $n$ large,
\[\|u(t_{n}+t)\|_{L^{10}_{t,x}([0, +\infty)\times \R^{3})}
=\|u\|_{L^{10}_{t,x}([t_{n}, +\infty))}\lesssim 1,
\]
contradicting \eqref{Eequal}.\end{proof}

\begin{lemma}\label{Compa}
There exists $c>0$ so that
\begin{equation}\label{InequeN}
{{F^{c}_{\infty}}}[u(t)]=8\|\nabla u(t)\|^{2}_{L^{2}}-8\|u(t)\|^{6}_{L^{6}}\geq c\delta(t)
\qtq{for all $t\geq 0$.}
\end{equation}
\end{lemma}
\begin{proof}
We recall that $E^{c}(W)=\frac{1}{3}\|\nabla W\|^{2}_{L^{2}}$.  By using $E(u)=E^{c}(W)$ we have 
\begin{equation}\label{Pohi22}
{{F^{c}_{\infty}}}[u(t)]=16\delta(t)-12\|u(t)\|^{4}_{L^{4}}.
\end{equation}
We claim that ${{F^{c}_{\infty}}}[u(t)]>0$ for all $t\geq 0$. Indeed, by using sharp Sobolev inequality \eqref{GI} and \eqref{PositiveP} 
 we deduce 
\begin{align}\nonumber
\|u\|^{6}_{L^{6}}&\leq  C^{6}_\text{GN}\|\nabla u\|_{L^2}^6
= \frac{1}{\|\nabla W\|^{4}_{L^{2}}}\| \nabla u\|^{6}_{L^{6}}\\ \label{IGU}
&=\( \tfrac{\|\nabla u\|_{L^2}}{\|\nabla W\|_{L^{2}}} \)^{4}\|\nabla u\|^{2}_{L^{2}}
<\|\nabla u\|^{2}_{L^{2}}.
\end{align}

Suppose now that \eqref{InequeN} were false. Then there exists  $\left\{t_{n}\right\}$ such that
\begin{equation}\label{Contra22}
{{F^{c}_{\infty}}}[u(t_{n})]\leq \tfrac{1}{n}\delta(t_{n}).
\end{equation}
We first observe that $\left\{\delta(t_{n})\right\}$ is bounded (cf. Lemma~\ref{GlobalS}). Now, we show that
\[
\delta(t_{n})\to 0 \quad \text{as $n\to\infty$}.
\]
In fact ${{F^{c}_{\infty}}}[u(t_{n})]\to 0$ as $n\to \infty$ by \eqref{Contra22}, so that \eqref{IGU} implies
\begin{align*}
	\|\nabla u(t_n)\|_{L^2}
		\rightarrow \|\nabla W\|_{L^2},
\end{align*}
i.e.  $\delta(t_{n})\to 0$ as $n\to\infty$.  By Proposition \ref{Modilation11} we have
\[
\| u(t_n)\|^{4}_{L^4}\lesssim \delta(t_{n})^{2}\leq \tfrac{1}{12}\delta(t_{n})\quad \text{for $n$ large}.
\]
Thus, by using \eqref{Pohi22} and \eqref{Contra22} we get
\[
15\delta(t_{n})\leq \tfrac{1}{n}\delta(t_{n})\quad \text{for $n$ large},
\]
which is a contradiction because $\delta(t)>0$ for all $t\geq 0$.
\end{proof}

Now we show the existence of a sequence $t_n\to\infty$ so that $\delta(t_n)\to 0$.

\begin{proposition}\label{BoundedUN}
There exists a sequence $t_{n}\to\infty$ such that
\[
\lim_{n\to \infty}\delta(t_{n})=0.
\]
\end{proposition}
\begin{proof} First we show that there exists $c>0$ such that
\begin{align}\label{Bla}
	\lambda(t)\geq c\qtq{for all $t\geq 0$.}
	\end{align}
Indeed, if \eqref{Bla} is false, there exists a sequence $\left\{t_{n}\right\}$ such that $\lambda(t_{n})\to 0$. 
From compactness in $\dot{H}^{1}(\R^{3})$ (cf. \eqref{CompacNew}) we deduce that there exists $\phi\in \dot{H}^{1}(\R^{3})$
so that
\begin{align}\label{CopC}
	u_{\lambda_{n}}(x):=\lambda(t_{n})^{-\frac{1}{2}}u(t_{n}, \tfrac{x}{\lambda(t_{n})})\to \phi \qtq{in $\dot{H}^{1}(\R^{3})$.}
\end{align}
Notice also that
\begin{align}\label{CNl}
	\|u_{\lambda_{n}}\|^{2}_{L^{2}}=\lambda(t_{n})^{2}\|u_{0}\|^{2}_{L^{2}}\to 0 \qtq{as $n \to \infty$.}
\end{align}
In particular, the sequence $\left\{u_{\lambda_{n}}\right\}$ is bounded in ${H}^{1}(\R^{3})$, $u_{\lambda_{n}}\rightharpoonup \phi$ in 
${H}^{1}(\R^{3})$ and 
\[
\|\phi\|^{2}_{L^{2}}\leq \liminf_{n\to \infty}\|u_{\lambda_{n}}\|^{2}_{L^{2}}=0,
\]
which yields $\phi\equiv 0$. But then, by \eqref{CopC} we deduce that 
$\|u(t_{n})\|^{2}_{\dot{H}^{1}}=\|u_{\lambda_{n}}\|^{2}_{\dot{H}^{1}}\to 0$. As $\|u(t_{n})\|^{2}_{L^{2}}=\|u_{0}\|^{2}_{L^{2}}$, it follows that
$\|u(t_{n})\|^{4}_{L^{4}}\to 0$ and 
\[
E^{c}(W)=E(u_{0})=\lim_{n\to \infty} E(u(t_{n}))=0,
\]
a contradiction. Thus \eqref{Bla} holds.

Next, given $T>0$ and $\epsilon>0$, we will prove that there exists $\rho_{\epsilon}>0$ so that
\begin{equation}\label{NewVirial}
\tfrac{1}{T}\int^{T}_{0}\delta(t)\,dt\lesssim\epsilon
+\tfrac{1}{T}\big(\sup_{t\in [0,T]}\tfrac{\rho_{\epsilon}}{\lambda(t)}\big)^{2}\|u\|^{2}_{L^{\infty}_{t}\dot{H}^{1}_{x}}.
\end{equation}
To prove this, we fix $R>0$, which will be determined later. Since $|\nabla w_{R}| \lesssim R^2|x|^{-1}$, 
\[
|I_{R}[{u}](t)|\lesssim R^{2}\|u\|^{2}_{L^{\infty}_{t}\dot{H}^{1}_{x}}
\]
(cf. Lemma~\ref{VirialIden}). From  Lemmas~\ref{VirialIden} and \ref{Compa}, we can write 
\begin{equation}\label{DecompV}
\tfrac{d}{dt}I_{R}[u(t)]=F^{c}_{\infty}[u(t)]+(F_{R}[u(t)]-F^{c}_{\infty}[u(t)])\geq c\delta(t)+(F_{R}[u(t)]-F^{c}_{\infty}[u(t)]).
\end{equation}
We now write
\begin{align}\label{F11}
F_{R}[u(t)]-F^{c}_{\infty}[{u}(t)]&=\int_{|x|\geq R}\big[(- \Delta \Delta w_{R})|u|^{2}
+4\RE \overline{u_{j}} u_{k} \partial_{jk}[w_{R}]-8|\nabla u|^{2}\big]dx\\ \label{F22}
&-\tfrac{4}{3}\int_{|x|\geq R}\Delta[w_{R}(x)]|u|^{6}dx
+8\int_{|x|\geq R}|u|^{6}dx
\\ \label{F33}
&+\int_{\R^{3}}\Delta[w_{R}(x)]|u|^{4}dx.
\end{align}
By compactness in $\dot{H}^{1}(\R^{3})$ (see \eqref{CompacNew}) and conservation of mass, we deduce that
\begin{align}\label{epeq}
	\sup_{t\in \R}\int_{|x|>\tfrac{\rho_{\epsilon}}{\lambda(t)}}\big[|\nabla u|^{2} +|u|^{6}+|u|^{4} +\tfrac{|u|^{2}}{|x|^{2}}\big](t,x)dx  \ll \epsilon
\end{align}
for large enough $\rho_\eps>0$.  Given $T>0$, we now choose
\[
R:=\sup_{t\in [0,T]}\frac{\rho_{\epsilon}}{\lambda(t)},\qtq{so that}
\left\{|x|\geq R\right\}\subset \left\{|x|\geq \tfrac{\rho_{\epsilon}}{\lambda(t)}\right\}\qtq{for all}t\in[0,T].
\]
It follows that
\begin{equation}\label{smallI}
|\eqref{F11}|+|\eqref{F11}|<\epsilon \qtq{for all $t\in [0, T]$.}
\end{equation}
Next, for \eqref{F33}, we use \eqref{epeq} and the fact that $|\Delta[w_{R}](x)|\lesssim 1$ to get
\[
\eqref{F33}
=6\int_{|x|\leq R}|u|^{4}dx+\int_{|x|\geq R}\Delta[w_{R}(x)]|u|^{4}dx\geq -\epsilon
 \qtq{for all $t\in [0, T]$.}
\]
Integrating \eqref{DecompV} on $[0, T]$ and applying \eqref{smallI} and the inequality above now yields
\[
\int^{T}_{0}\delta(t) \,dt\lesssim \big(\sup_{t\in [0,T]}\tfrac{\rho_{\epsilon}}{\lambda(t)}\big)^{2}\|u\|^{2}_{L^{\infty}_{t}\dot{H}^{1}_{x}}+\epsilon T.
\]
Recalling the lower bound on $\lambda$,  we deduce that
\[
\tfrac{1}{T}\int^{T}_{0}\delta(t)dt\lesssim\epsilon
+\tfrac{\rho^{2}_{\epsilon}}{T}.
\]
As $\eps>0$ was arbitrary, we can therefore find sequences $T_n\to\infty$ and $t_{n}\in [\frac{1}{2}T_{n}, T_{n}]$ so that $\delta(t_{n})\to 0$  as $n\to \infty$.
\end{proof}

In order to capitalize on the preceding proposition, we will now prove a sharpened virial estimate in which the right-hand side is controlled by $\delta(\cdot)$.

\begin{lemma}\label{Lemma111}
For $\delta_{1}\in (0, \delta_{0})$ sufficiently small, there exists a constant $C=C(\delta_{1})>0$ such that for any interval $[t_{1}, t_{2}]\subset [0, \infty)$,
\begin{equation}\label{BoundT11}
\int^{t_{2}}_{t_{1}}\delta(t)\,dt\leq C\sup_{t\in[t_{1},t_{2}]}\tfrac{1}{\lambda(t)^{2}}
\left\{\delta(t_{1})+\delta(t_{2})\right\}.
\end{equation}
\end{lemma}
\begin{proof}
Fix  $\delta_{1}\in (0, \delta_{0})$ and $R>1$ to be determined below. We use the localized virial identities (cf. Lemma~\ref{VirialModulate}) with $\chi(t)$ satisfying
\[
\chi(t)=
\begin{cases}
1& \quad \delta(t)<\delta_{1} \\
0& \quad \delta(t)\geq \delta_{1}.
\end{cases}
\]
From  Lemmas~\ref{VirialModulate} and \ref{Compa} we  deduce
\begin{equation}\label{VirilaX11}
\tfrac{d}{dt}I_{R}[u(t)]=F^{c}_{\infty}[u(t)]+\EE(t)\geq c\delta(t)+\EE(t)
\end{equation}
with
\begin{equation}\label{Error111}
\EE(t)=
\begin{cases}
F_{R}[u(t)]-F^{c}_{\infty}[u(t)]& \quad  \text{if $\delta(t)\geq \delta_{1}$}, \\
F_{R}[u(t)]-F^{c}_{\infty}[u(t)]-\K[t]& \quad \text{if $\delta(t)< \delta_{1}$},
\end{cases}
\end{equation}
where
\begin{equation}\label{Error222}
\K(t)=F^{c}_{R}[e^{ i \theta(t)}\lambda(t)^{\frac{1}{2}}W(\lambda(t)x)]
-F^{c}_{\infty}[e^{ i \theta(t)}\lambda(t)^{\frac{1}{2}}W(\lambda(t)x)].
\end{equation}

We now make two claims that together will imply the desired result. 

\textbf{{Claim I.} }For $R>1$,  we have
\begin{align}\label{EstimateV111}
	|I_{R}[u(t_{j})]|\lesssim \tfrac{R^{2}}{\delta_{1}}\delta(t_{j}) \quad& \text{if $\delta(t_{j})\geq \delta_{1}$ for $j=1$, $2$},\\
	\label{EstimateV221}
	|I_{R}[u(t_{j})]|\lesssim R^{2} \delta(t_{j}) \quad &\text{if $\delta(t_{j})< \delta_{1}$ for $j=1$, $2$}.
\end{align}

\textbf{{Claim II.}} For $\epsilon>0$, there exists $\rho_{\epsilon}=\rho(\epsilon)>0$ so that if
\begin{equation}\label{J-defR}
R:=\rho_{\epsilon}\sup_{t\in [t_{1}, t_{2}]}\tfrac{1}{\lambda(t)},
\end{equation}
then
\begin{align}\label{EstimateE111}
\EE(t)\geq-\epsilon\geq- \tfrac{\epsilon}{\delta_{1}}\delta(t) \quad &\text{uniformly for $t\in [t_{1}, t_{2}]$ such that
	$\delta(t)\geq \delta_{1}$},\\
\label{EstimateE221}
|\EE(t)|\leq (\epsilon+\delta_{1}) \delta(t)\quad &\text{uniformly for $t\in [t_{1}, t_{2}]$ such that $\delta(t)< \delta_{1}$}.
\end{align}

Integrating inequality \eqref{VirilaX11} on $[t_{1}, t_{2}]$ and applying  the estimates \eqref{EstimateV111}, \eqref{EstimateV221}, \eqref{EstimateE111} and \eqref{EstimateE221} we deduce
\[
\int^{t_{2}}_{t_{1}}\delta(t)\,dt\lesssim 
\tfrac{\rho^{2}_{\epsilon}}{\delta^{2}_{1}}\sup_{t\in [t_{1}, t_{2}]}\tfrac{1}{\lambda(t)^{2}}(\delta(t_{1})+\delta(t_{2}))
+(\tfrac{\epsilon}{\delta_{1}}+\epsilon+\delta_{1})\int^{t_{2}}_{t_{1}}\delta(t)\,dt.
\]
For $\delta_{1}\in (0, \delta_{0})$ small, and  choosing $\epsilon=\epsilon(\delta_{1})$ sufficiently small yields \eqref{BoundT11}. 
\end{proof}

Thus it remains to establish the above claims.
\begin{proof}[{Proof of Claim I}] Assume that $\delta(t_{j})\geq \delta_{1}$. Then
\[
|I_{R}[u(t_{j})]|
\lesssim R^{2}\|u\|^{2}_{L^{\infty}_{t}\dot{H}^{1}}\lesssim_{W} \tfrac{R^{2}}{\delta_{1}}\delta(t_{j}).
\]
This proves \eqref{EstimateV111}. 

We now write $W_{[\lambda(t)]}:=\lambda(t)^{\frac{1}{2}}W(\lambda(t)\cdot)$. If  $\delta(t_{j})< \delta_{1}$, then by using the fact that $W$ is real we deduce
\begin{align*}
|I_{R}[u(t_{j})]|&\leq \left|2\IM\int_{\R^{3}}\nabla w_{R}(\overline{u} \nabla u-e^{-i\theta(t_{j})}W_{[\lambda(t_{j})]} 
\nabla [ e^{i\theta(t_{j})}W_{[\lambda(t_{j})]}])
dx\right|\\
&\lesssim
R^{2}[\|u\|_{L^{\infty}_{t}\dot{H}^{1}_{x}}+\|W\|_{\dot{H}^{1}}]
\|u(t_{j})-e^{i\theta (t_{j})}W_{[\lambda(t_{j})]}\|_{\dot{H}^{1}}\\
&\lesssim_{W}R^{2} \|g(t_{j})\|_{\dot{H}^{1}}   \lesssim_{W}R^{2} \delta(t_{j}),
\end{align*}
where in the last inequality we have used estimate \eqref{Estimatemodu}.
\end{proof}

\begin{proof}[{Proof of Claim II}]
First, assume $\delta(t)\geq \delta_{1}$. By compactness \eqref{CompacNew}, we have that for each $\epsilon>0$, there exists 
$\rho_{\epsilon}=\rho (\epsilon)>0$  such that
\begin{equation}\label{CompactAgain2}
	\sup_{t\in \R}\int_{|x|>\tfrac{\rho_{\epsilon}}{\lambda(t)}}[|\nabla u|^{2} +|u|^{6}+|u|^{4} +\tfrac{|u|^{2}}{|x|^{2}} ](t,x)dx  \ll \epsilon.
\end{equation}
Defining $R$ as in \eqref{J-defR} and using the same argument developed above in Step 2 of Lemma~\ref{BoundedUN}, we deduce that
\[
\EE(t)= F_{R}[u(t)]-F_{\infty}[u(t)]\geq -\epsilon\geq -\tfrac{\epsilon}{\delta_{1}} \delta(t)
\quad \text{for every $t\in [t_{1}, t_{2}]$ with $\delta(t)\geq \delta_{1}$},
\]
which yields \eqref{EstimateE111}.

Now we consider the case  $\delta(t)< \delta_{1}$. To simplify notation, we set  $W(t)=e^{i\theta(t)}\lambda(t)^{\frac{1}{2}}W(\lambda(t)x)$. For $t\in[t_1,t_2]$ with $\delta(t)<\delta_1$, we may write 
\begin{align}\label{Decomp111}
&\EE(t)\nonumber\\
&=-8\int_{|x|\geq R}[(|\nabla u|^{2}-(|\nabla W(t)|^{2})]\,dx
\quad+8\int_{|x|\geq R}[|u|^{6}-|W(t)|^{6}]\,dx\\\label{Decomp331}
&\quad+\int_{|x|\geq R}(-\Delta \Delta w_{R})[|u|^{2}-|W(t)|^{2}]\,dx
-\tfrac{4}{3}\int_{|x|\geq R}\Delta[w_{R}(x)](|u|^{6}-|W(t)|^{6})\,dx\\\label{Decomp551}
&\quad+4\RE\int_{|x|\geq R}[\overline{u_{j}} u_{k}-\overline{\partial_{j}W(t)} \partial_{k}W(t)]\partial_{jk}[w_{R}(x)]dx\\ \label{Decomp661}
&\quad +\int_{\R^{3}}\Delta[w_{R}(x)]|u|^{4}\,dx.
\end{align}

As $|\Delta \Delta w_{R}|\lesssim 1/|x|^{2}$, $|\partial_{jk}[w_{R}]|\lesssim 1$ and $|\Delta[w_{R}]|\lesssim 1$,  we see that \eqref{Decomp111}--\eqref{Decomp551} can be estimated by terms of the form 
\begin{align*}
&\bigl\{\|u(t)\|_{\dot{H}_{x}^{1}(|x|\geq R)}+\|W(t)\|_{\dot{H}_{x}^{1}(|x|\geq R)} +
\|u(t)\|^{5}_{L_{x}^{6}(|x|\geq R)}\\
& \quad+\|W(t)\|^{5}_{L_{x}^{6}(|x|\geq R)}\bigr\}\|g(t)\|_{\dot{H}_{x}^{1}}+\bigl\{\|\tfrac{u(t)}{|x|}\|_{L^{2}(|x|\geq R)}+\|\tfrac{W(t)}{|x|}\|_{L^{2}(|x|\geq R)}\bigr\}\|g(t)\|_{\dot{H}_{x}^{1}},
\end{align*}
where $ g(t)=e^{-i\theta(t)}[{u}(t)-W(t)]$. Moreover, since 
\[
\|W(t)\|_{\dot{H}_{x}^{1}(|x|\geq R)} \sim \tfrac{1}{R^{1/2}}\qtq{and}\|W(t)\|_{L_{x}^{6}(|x|\geq R)} \sim \tfrac{1}{R^{1/2}},
\]
it follows from \eqref{CompactAgain2} and \eqref{Estimatemodu} that
\[
|\eqref{Decomp111}|+|\eqref{Decomp331}|+|\eqref{Decomp551}|\lesssim \epsilon \delta(t).
\]
Finally, by Proposition~\ref{Modilation11} we see that
\[
\left|\int_{\R^{3}}\Delta[w_{R}(x)]|u|^{4}dx\right|\lesssim \delta(t)^{2}\lesssim \delta_{1} \delta(t).
\]
Putting together the estimates above yields \eqref{EstimateE221}.\end{proof}

\begin{remark}\label{GrwI}
There exists a constant $C>0$ so that for $t\geq 0$,
\begin{align}\label{BdeC}
\int^{+\infty}_{t}\delta(s)\,ds\leq Ce^{-ct}.
\end{align}
Indeed, by Step 1 of Proposition~\ref{BoundedUN} we see that there exists $c>0$ so that $\lambda(t)\geq c$ for all $t\geq 0$.
Thus, from Lemma~\ref{Lemma111} we deduce that there exists $C>0$ so that
\[
\int^{t_{n}}_{t}\delta(s)\,ds \leq C[\delta(t)+\delta(t_{n})],
\]
where a sequence $\left\{t_{n}\right\}$ is given in Proposition~\ref{BoundedUN}. Let $n\to +\infty$. Taking the limit as $n\to\infty$, we then obtain \eqref{BdeC} from Gronwall's inequality.\end{remark}

We next show that we may use $\delta$ to control the variation of the spatial scale.

\begin{proposition}\label{Spatialcenter}
Let $[t_{1}, t_{2}]\subset(0, \infty)$ with $t_{1}+\tfrac{1}{\lambda(t_{1})^{2}}< t_{2}$. Then
\begin{equation}\label{BoundCenter1}
\left|\tfrac{1}{\lambda(t_{2})^{2}}-\tfrac{1}{\lambda(t_{1})^{2}}\right|\leq \widetilde{C}\int^{t_{2}}_{t_{1}}\delta(t)\,dt.
\end{equation}
\end{proposition}

\begin{proof}
\textsl{Step 1.}  First we show that there exists a constant $C_{1}$ such that
\begin{equation}\label{step11}
\tfrac{\lambda(s)}{\lambda(t)}+\tfrac{\lambda(t)}{\lambda(s)}\leq C_{1} \quad \text{for all $t$, $s\geq 0$ such that $|t-s|\leq \tfrac{1}{\lambda(s)^{2}} $}.
\end{equation}
Suppose instead that there exist two sequences $s_{n}$ and  $t_{n}$  so that
\begin{align}\label{CBound}
|t_{{n}}-s_{n}|\leq \tfrac{1}{\lambda(s_{n})^{2}} \qtq{but}
\tfrac{\lambda(s_{n})}{\lambda(t_{n})}+\tfrac{\lambda(t_{n})}{\lambda(s_{n})}\to \infty.
\end{align}
Passing to a subsequence, one can suppose that
\[
\lim_{n\to \infty}\lambda(s_{n})^{2}(t_{n}-s_{n})=\tau_{0}\in[-1, 1].
\]

Now, since $\lambda(t)\gtrsim 1$ for all $t\geq 0$ (cf. \eqref{Bla}), by \eqref{CBound} we see that
\[
\lambda(s_{n})+\lambda(t_{n}) \to \infty \qtq{as $n\to \infty$.}
\]

\textbf{Scenario I.} Suppose  $\lambda(s_{n}) \to \infty$ as $n\to\infty$. From \eqref{CBound} we deduce $|t_{n}-s_{n}|\to 0$ as $n\to \infty$. Then by the local 
theory for \eqref{NLS} and $H^1$-boundedness we may obtain that 
\[
\|u(t_{n})-u(s_{n})\|_{\dot{H}^{1}}\to 0
\qtq{as $n\to \infty$. }
\]

By a change of variables, this yields
\[
\|\lambda(s_{n})^{-1/2}u(t_{n}, \tfrac{\cdot}{\lambda(s_{n})})-\lambda(s_{n})^{-1/2}u(s_{n}, \tfrac{\cdot}{\lambda(s_{n})})\|_{\dot{H}^{1}}\to 0
\qtq{as $n\to \infty$.}
\]
Now, by compactness, $\psi\neq 0$ so that
\begin{align}\label{CvC}
	\lambda(s_{n})^{-1/2}u(t_{n}, \tfrac{\cdot}{\lambda(s_{n})}) \to \psi 
	\qtq{in $\dot{H}^{1}$ as $n\to \infty$}
\end{align}
along some subsequence. Writing $w_{n}(x):=\lambda(s_{n})^{-1/2}u(t_{n}, \tfrac{\cdot}{\lambda(s_{n})})$, we may apply compactness again to deduce 
\begin{align}\label{2dce}
\bigl(\tfrac{\lambda(s_{n})}{\lambda(t_{n})}\bigr)^{\frac{1}{2}} w_{n}(\tfrac{\lambda(s_{n})}{\lambda(t_{n})}x)=\lambda(t_{n})^{-1/2}u(t_{n}, \tfrac{\cdot}{\lambda(t_{n})})\to \zeta \qtq{in $\dot{H}^{1}$ as $n\to \infty$}
\end{align}
for some $\zeta$, which is necessarily nonzero due to conservation of energy. Finally, by \eqref{2dce} we see that there exists $a,C>0$ so that
\[
\int_{a\,\tfrac{\lambda(s_{n})}{\lambda(t_{n})}\leq |x|\leq 2a\,\tfrac{\lambda(s_{n})}{\lambda(t_{n})}}
|\nabla w_{n}(x)|^{2}dx\geq C\qtq{for all large}n.
\]
However, by \eqref{CvC}, \eqref{CBound} and convergence dominated theorem we deduce that
\[
\int_{a\,\tfrac{\lambda(s_{n})}{\lambda(t_{n})}\leq |x|\leq 2a\,\tfrac{\lambda(s_{n})}{\lambda(t_{n})}}
|\nabla w_{n}(x)|^{2}dx\to 0
\]
as $n\to \infty$, which is a contradiction.

\textbf{Scenario II.} Suppose $\lambda(t_{n})\to \infty$. We will show that $\lambda(s_{n})\to \infty$, as well.

Suppose instead that $\lambda(s_{n})$  is bounded.  By \eqref{CBound} we can assume that there exists $\tau\in \R$ so that
$t_{n}-s_{n}\to \tau$.  Since  $\lambda(s_{n})$  is bounded, by compactness \eqref{CompacNew}, we see that, up to subsequence,  $u(s_{n})\to v_{0}$ in $\dot{H}^{1}$ as $n\to +\infty$ for some $v_{0}\in H^{1}$. By conservation of mass, we also have that $u(s_{n})\to v_{0}$ in $L^{4}$. In particular, $E(v_{0})=E^{c}(W)$. Moreover, by \eqref{IneW} we have that  $\delta(v_{0})\geq c>0$.

From Remark~\ref{FirRe} and Lemma~\ref{GlobalS} we deduce that the solution $v$ to \eqref{NLS} with initial data $v_{0}$ is global and satisfies $\delta(v(t))>0$ for all $t\geq 0$. In addition, Lemma~\ref{stabi} implies that $u(s_{n}+\tau)\to v(\tau)$ in $\dot{H}^{1}$ and so 
$\delta(u(s_{n}+\tau))\geq c_{0}$ for $n$ large. From the local 
theory for \eqref{NLS} we have $\|u(t_{n})-u(s_{n}+\tau)\|_{\dot{H}^{1}}\to 0$
as $n\to +\infty$. This implies $\delta(t_{n})\geq c_{0}$ for some $c_{0}>0$ and all $n$ large.

Now, as $\delta(t_{n})\geq c_{0}$, from Lemma~\ref{DeltaZero}  we deduce that $t_{n} \to t_{\ast}\in [0, \infty)$ along some subsequence.  By continuity of the flow and compactness, we see that
\begin{align}\label{Cqe}
	& u(t_{n}) \to u(t_{\ast})\neq 0 \qtq{ strongly in $H^{1}$,}\\  \label{lmn}
	& \lambda(t_{n})^{-1/2}u(t_{n}, \tfrac{\cdot}{\lambda(t_{n})}) \to \zeta\neq 0
	\qtq{ strongly in $\dot{H}^{1}$.}
\end{align}
As $\lambda(t_{n})\to \infty$, \eqref{Cqe} and \eqref{lmn} together yield a contradiction. Thus $\lambda(s_{n})\to \infty$ as $n\to \infty$ and hence Scenario I implies \eqref{step11}.

\textsl{Step 2.}
There exists $\delta_{1}>0$,  so that for any $T\geq 0$, either
\begin{equation}\label{MinMax}
\inf_{t\in [T, T+\tfrac{1}{\lambda(T)^{2}}]}\delta(t)\geq \delta_{1} \quad \text{or}\quad
\sup_{t\in [T,T+\tfrac{1}{\lambda(T)^{2}}]}\delta(t)<\delta_{0}.
\end{equation}
Suppose instead that there exist $t_{n}^{\ast}\geq 0$ and
$t_{n}$, $t^{\prime}_{n}\in  [t_{n}^{\ast}, t_{n}^{\ast}+\tfrac{1}{\lambda(t_{n}^{\ast})^{2}}]$ so that
\begin{align}\label{ContraStep2}
&\delta(t_{n})\to 0 \quad \text{and}\quad \delta(t^{\prime}_{n})\geq \delta_{0}.
\end{align}

Notice that $t_{n}\to \infty$. Indeed, if $t_{n}\leq C$ for all $n\in \N$, then (up to a  subsequence)  $t_{n}\to a\in [0, \infty)$. By continuity and \eqref{ContraStep2}, we derive that $\delta(a)=0$, which contradicts the fact that  $\delta(t)>0$ for all $t\geq 0$.  Thus, as $t_{n}\to \infty$ and $\delta(t_{n})\to 0$, Lemma~\ref{DeltaZero} implies that
\begin{align}\label{Inll}
\lambda(t_{n})\to \infty  \quad\text{as\ } {n\to \infty}.
\end{align}

On the other hand, note that $\lambda(t^{\prime}_{n})$ is bounded. Indeed, this follows from Lemma~\ref{DeltaZero} and \eqref{ContraStep2} if $t^{\prime}_{n}$ is unbounded, and by Step 1 above if $t^{\prime}_{n}$ is bounded. 

Now, since  $\lambda(t^{\prime}_{n})$ is bounded, Step 1 shows that $\lambda(t_{n}^{\ast})$ is bounded,  which implies (again by Step 1) that $\lambda(t_{n})$ is bounded, contradicting \eqref{Inll}.

\textsl{Step 3.} We show there exists a constant $C>0$ so that
\begin{align}\label{BoundCenter}
0\leq t_{1}\leq t_1'\leq t_2'\leq t_{2}=t_{1}+\tfrac{1}{C^{2}_{1}\lambda(t_{1})^{2}}
\Rightarrow 
\left|\tfrac{1}{\lambda(t_2')^{2}}-\tfrac{1}{\lambda(t_1')^{2}}\right|\leq C\int^{t_{2}}_{t_{1}}\delta(t)\,dt,
\end{align}
for some $C>0$. Here $C_{1}\geq 1$ is the constant given in Step 1.

By Step 2, we can assume that $\sup_{t\in [t_{1}, t_{2}]}\delta(t)<\delta_{0}$ or
$\inf_{t\in [t_{1}, t_{2}]}\delta(t)\geq \delta_{1}$. In the former case, we deduce \eqref{BoundCenter} via the estimate  
\[
\left|\tfrac{\lambda^{\prime}(t)}{\lambda(t)^{3}}\right|\lesssim \delta(t)\qtq{for}\delta(t)<\delta_{0}
\]
(cf. \eqref{EstimLaD}).  In the latter case, we see that $\int^{t_{2}}_{t_{1}}\delta_{1}\,dt\geq \int^{t_{2}}_{t_{1}}\delta(t)\,dt$ and
\[
|t_1'-t_2'|\leq \tfrac{1}{C^{2}_{1}\lambda(t_{1})^{2}}\leq \tfrac{1}{\lambda(t_1')^{2}}.
\]
 Thus, by Step 1 we obtain
\[
\left|\tfrac{1}{\lambda(t_2')^{2}}-\tfrac{1}{\lambda(t_1')^{2}}\right|
\leq\tfrac{2 C^{2}_{1}}{\lambda(t_1')^{2}}
\leq \tfrac{2 C^{4}_{1}}{\lambda({t_{1}})^{2}}
=2C^{5}_{1}|t_{2}-t_{1}|\leq \tfrac{2C^{5}_{1}}{\delta_{1}}\int^{t_{2}}_{t_{1}}\delta(t)dt.
\]
Combining the inequalities above yields \eqref{BoundCenter}.
\end{proof}

We now rule out the infinite-time blowup case:

\begin{proof}[{Proof of Theorem~\ref{Tfinito22}}] Suppose $u$ is a solution as in Proposition~\ref{Compacness11} with $T^{\ast}=\infty$.  In particular, $E(u_{0})=E^{c}(W)$, $\|\nabla u_{0}\|^{2}_{L^{2}}<\|\nabla W\|^{2}_{L^{2}}$ and 
 there exists a function  $0<\lambda(t)<\infty$ for all $t\geq0$, such that the set
 \begin{equation*}
\left\{\lambda(t)^{-\frac{1}{2}}u(t, \tfrac{x}{\lambda(t)}): t\in [0, \infty)\right\} \quad
\text{is pre-compact in $\dot{H}^{1}(\R^{3})$}.
\end{equation*}

From Proposition~\ref{BoundedUN}, there exists an increasing sequence $t_{n}\to\infty$ with $\delta(t_{n})\to 0$. In particular, we can choose $t_{n}$ so that
\[
\delta(t_{0})\leq \tfrac{1}{2C_{0}} \qtq{and} \delta(t_{n})\to 0,
\]
where $C_{0}:=C\cdot\widetilde{C}$, with $C$ and $\widetilde{C}$ the constants in \eqref{BoundT11} and \eqref{BoundCenter1}. 

Consider $0\leq a\leq b$.  For $n$ large we see that $b+\tfrac{1}{\lambda(b)^{2}}<t_{n}$. By estimates \eqref{BoundT11} and \eqref{BoundCenter1}, we find
\[
\left|\tfrac{1}{\lambda(b)^{2}}-\tfrac{1}{\lambda(t_{n})^{2}}\right|\leq C_{0}
\sup_{a\leq t\leq t_{n}}\bigl(\tfrac{1}{\lambda(t)^{2}}\bigr)\left\{\delta(a)+\delta(t_{n})  \right\}.
\]
Sending $n\to\infty$, it follows that 
\[
\sup_{t\geq a}\tfrac{1}{\lambda(t)^{2}}\leq C_{0}\delta(a)\sup_{t\geq a}\tfrac{1}{\lambda(t)^{2}}.
\]
Choosing $a=t_{0}$, we get
\[
\sup_{t\geq t_{0}}\tfrac{1}{\lambda(t)^{2}}=0.
\]
Lemma~\ref{Lemma111} then implies that $\delta(t)=0$ for all $t\geq t_{0}$, a contradiction.\end{proof}

Proposition~\ref{Compacness11} and Theorems~\ref{Tfinito}--\ref{Tfinito22} now yield Theorem~\ref{Th2}(i).

\bibliography{bibliografia}

\end{document}